\documentclass[11pt]{amsart}
\usepackage{amsfonts}
\usepackage{amsmath}
\usepackage{amssymb}
\usepackage{txfonts}
\usepackage{color}
\usepackage{graphicx}
\usepackage{xspace}
\usepackage{axodraw}
 \font \eightrm=cmr8

 \newcommand{\nc}{\newcommand}

 \nc{\butcher}{{\scriptstyle\ \circleright\ }}
 
\nc{\surj}{\to\hskip -3mm \to}

\newtheorem{thm}{Theorem}

\newtheorem{cor}[thm]{Corollary}

\newtheorem{lem}[thm]{Lemma}

\newtheorem{rmk}[thm]{Remark}

\def\racine{{\scalebox{0.3}{
\begin{picture}(12,12)(38,-38)
\SetWidth{0.5} \SetColor{Black} \Vertex(45,-30){6}
\end{picture}
}}}
  
 \def\arbrea{\,{\scalebox{0.15}{ 
  \begin{picture}(8,55) (370,-248)
    \SetWidth{2}
    \SetColor{Black}
    \Line(374,-244)(374,-200)
    \Vertex(374,-197){9}
    \Vertex(375,-245){12}
  \end{picture}
}}\,}

 \def\arbreba{\,{\scalebox{0.15}{ 
\begin{picture}(8,106) (370,-197)
    \SetWidth{2}
    \SetColor{Black}
    \Line(374,-193)(374,-149)
    \Vertex(374,-146){9}
    \Vertex(375,-194){12}
    \Line(374,-142)(374,-98)
    \Vertex(374,-95){9}
  \end{picture}
}}\,}

 \def\arbrebb{\,{\scalebox{0.15}{ 
  \begin{picture}(48,48) (349,-255)
    \SetWidth{2}
    \SetColor{Black}
    \Vertex(375,-252){12}
    \Line(376,-250)(395,-215)
    \Line(373,-251)(354,-214)
    \Vertex(353,-211){9}
    \Vertex(395,-213){9}
  \end{picture}
}}}

\def\arbreca{\,{\scalebox{0.15}{
\begin{picture}(8,156) (370,-147)
    \SetWidth{2}
    \SetColor{Black}
    \Line(374,-143)(374,-99)
    \Vertex(374,-96){9}
    \Vertex(375,-144){12}
    \Line(374,-92)(374,-48)
    \Vertex(374,-45){9}
    \Line(374,-42)(374,2)
    \Vertex(374,5){9}
  \end{picture}
}}\,}

\def\arbrecb{\,{\scalebox{0.15}{
\begin{picture}(48,94) (349,-255)
\SetWidth{2}
\SetColor{Black}
\Line(376,-204)(395,-169)
\Line(373,-205)(354,-168)
\Vertex(353,-165){9}
\Vertex(395,-167){9}
\Vertex(374,-205){9}
\Line(374,-246)(374,-209)
\Vertex(374,-252){12}
\end{picture}}}\,}

\def\arbrecc{\,{\scalebox{0.15}{
 \begin{picture}(48,98) (349,-205)
    \SetWidth{2}
    \SetColor{Black}
    \Vertex(375,-202){12}
    \Line(376,-200)(395,-165)
    \Line(373,-201)(354,-164)
    \Vertex(353,-161){9}
    \Vertex(395,-163){9}
    \Line(353,-160)(353,-113)
    \Vertex(353,-111){9}
  \end{picture}
}}\,}

\def\arbrecd{\,{\scalebox{0.15}{
\begin{picture}(48,52) (349,-251)
    \SetWidth{2}
    \SetColor{Black}
    \Vertex(375,-248){12}
    \Line(376,-246)(395,-211)
    \Line(373,-247)(354,-210)
    \Vertex(353,-207){9}
    \Vertex(395,-209){9}
    \Line(375,-247)(375,-206)
    \Vertex(376,-203){9}
  \end{picture}
 }}\,}

\def\arbrece{\,{\scalebox{0.2}{
  \begin{picture}(45,78) (7,-37)
    \SetWidth{1.0}
    \SetColor{Black}
    \Line(30,-27)(15,3)
    \SetWidth{0.5}
    \Vertex(30,-27){9.9}
    \SetWidth{1.0}
    \Line(30,-27)(45,3)
    \SetWidth{0.5}
    \Vertex(15,3){7.07}
    \Vertex(45,3){7.07}
    \SetWidth{1.0}
    \Line(45,3)(45,33)
    \SetWidth{0.5}
    \Vertex(45,33){7.07}
  \end{picture}
}}\,}

\newcommand{\tree}{\hskip 0.8pc\scalebox{-0.3}{{\parbox{0.5pc}{
  \begin{picture}(30,45) (75,-60)
    \SetWidth{1.5}
    \SetColor{Black}
    \Line(90,-30)(75,-60)
    \Line(90,-30)(105,-60)
    \Line(90,-15)(90,-30)
  \end{picture}}}}}

\newcommand{\treeA}{\hskip 1.5pc\scalebox{-0.3}{{\parbox{0.5pc}{
   \begin{picture}(60,75) (75,-30)
    \SetWidth{1.5}
    \SetColor{Black}
    \Line(90,0)(75,-30)
    \Line(90,0)(105,-30)
    \Line(105,30)(90,0)
    \Line(105,30)(135,-30)
    \Line(105,45)(105,30)
  \end{picture}}}}}

\newcommand{\treeB}{\hskip 1.5pc\scalebox{-0.3}{{\parbox{0.5pc}{
  \begin{picture}(60,75) (75,-30)
    \SetWidth{1.5}
    \SetColor{Black}
    \Line(90,0)(75,-30)
    \Line(105,30)(135,-30)
    \Line(105,45)(105,30)
    \Line(120,0)(105,-30)
    \Line(105,30)(90,0)
  \end{picture}}}}}

\newcommand{\treeC}{\hskip 1.5pc\scalebox{-0.2}{{\parbox{0.5pc}{
 \begin{picture}(90,105) (75,-30)
    \SetWidth{2.9}
    \SetColor{Black}
    \Line(90,0)(75,-30)
    \Line(120,60)(90,0)
    \Line(90,0)(105,-30)
    \Line(105,30)(135,-30)
    \Line(120,60)(165,-30)
    \Line(120,75)(120,60)
  \end{picture}
}}}}

\newcommand{\treeD}{\hskip 1.5pc\scalebox{-0.2}{{\parbox{0.5pc}{
 \begin{picture}(90,105) (75,-30)
    \SetWidth{2.9}
    \SetColor{Black}
    \Line(90,0)(75,-30)
    \Line(120,60)(90,0)
    \Line(90,0)(105,-30)
    \Line(120,60)(165,-30)
    \Line(120,75)(120,60)
    \Line(150,0)(135,-30)
  \end{picture}
}}}}

\newcommand{\treeE}{\hskip 1.5pc\scalebox{-0.2}{{\parbox{0.5pc}{
 \begin{picture}(90,105) (75,-30)
    \SetWidth{2.9}
    \SetColor{Black}
    \Line(90,0)(75,-30)
    \Line(120,60)(90,0)
    \Line(120,60)(165,-30)
    \Line(120,75)(120,60)
    \Line(150,0)(135,-30)
    \Line(135,30)(105,-30)
  \end{picture}
}}}}

\newcommand{\treeF}{\hskip 1.5pc\scalebox{-0.2}{{\parbox{0.5pc}{
 \begin{picture}(90,105) (75,-30)
    \SetWidth{2.9}
    \SetColor{Black}
    \Line(90,0)(75,-30)
    \Line(120,60)(90,0)
    \Line(120,60)(165,-30)
    \Line(120,75)(120,60)
    \Line(135,30)(105,-30)
    \Line(120,0)(135,-30)
  \end{picture}
}}}}

\newcommand{\treeG}{\hskip 1.5pc\scalebox{-0.2}{{\parbox{0.5pc}{
 \begin{picture}(90,105) (75,-30)
    \SetWidth{2.9}
    \SetColor{Black}
    \Line(90,0)(75,-30)
    \Line(120,60)(90,0)
    \Line(120,60)(165,-30)
    \Line(120,75)(120,60)
    \Line(105,30)(135,-30)
    \Line(120,0)(105,-30)
  \end{picture}
}}}}

\newcommand{\treebplus}{\scalebox{0.3}{{\parbox{0.5pc}{
\begin{picture}(457,327) (72,-58)
    \SetWidth{3}
    \SetColor{Black}
    \Line(110,118)(187,41)
    \Line(188,39)(188,-10)
    \Line(188,39)(373,145)
    \DashLine(374,146)(424,173){7}
    \Line(262,82)(200,144)
    \Line(343,128)(290,181)
    \Line(461,197)(423,235)
    \Line(424,174)(529,237)
    \Text(95,118)[lb]{\Huge{\Black{$\Phi^{-1}(t_1)$}}}
    \Text(183,144)[lb]{\Huge{\Black{$\Phi^{-1}(t_2)$}}}
    \Text(272,180)[lb]{\Huge{\Black{$\Phi^{-1}(t_3)$}}}
    \Text(411,239)[lb]{\Huge{\Black{$\Phi^{-1}(t_n)$}}}
  \end{picture}
}}}}

  \def\colotree{\,{\scalebox{0.6}{  
 \begin{picture}(192,155) (99,-27)
    \SetWidth{1.9}
    \SetColor{Blue}
    \Line(240,107)(224,75)
    \Line(224,75)(208,43)
    \Line(208,43)(192,11)
    \Line(192,11)(192,-21)
    \SetColor{WildStrawberry}
    \Line(176,43)(192,11)
    \Line(176,75)(176,43)
    \Line(176,107)(176,75)
    \SetColor{Green}
    \Line(208,75)(208,43)
    \Line(208,107)(208,75)
    \SetColor{Gray}
    \Line(144,107)(176,75)
    \SetColor{BlueGreen}
    \Line(144,75)(176,43)
    \SetWidth{1.0}
    \SetColor{Black}
    \Vertex(144,75){4}
    \Vertex(144,107){4}
    \Vertex(176,43){4}
    \Vertex(176,75){4}
    \Vertex(176,107){4}
    \Vertex(208,107){4}
    \Vertex(208,75){4}
    \Vertex(208,43){4}
    \Vertex(192,11){4}
    \Vertex(224,75){4}
    \Vertex(240,107){4}
    \Vertex(192,-21){5}
    \SetWidth{1.9}
    \SetColor{Fuchsia}
    \Line(128,43)(192,-21)
    \SetWidth{1.0}
    \SetColor{Black}
    \Vertex(160,11){4}
    \Vertex(128,43){4}
    \Text(120,51)[lb]{\Large{\Black{$\ell^{(2)}$}}}
    \Text(240,115)[lb]{\Large{\Black{$\ell^{(4)}$}}}
    \Text(204,115)[lb]{\Large{\Black{$\ell^{(2)}$}}}
    \Text(173,115)[lb]{\Large{\Black{$\ell^{(3)}$}}}
    \Text(142,115)[lb]{\Large{\Black{$\ell^{(1)}$}}}
    \Text(137,82)[lb]{\Large{\Black{$\ell^{(1)}$}}}
  \end{picture}
}}}

\nc{\ignore}[1]{{}}
\nc{\mrm}[1]{{\rm #1}}
\nc{\dirlim}{\displaystyle{\lim_{\longrightarrow}}\,}
\nc{\invlim}{\displaystyle{\lim_{\longleftarrow}}\,}
\nc{\vep}{\varepsilon} \nc{\ep}{\epsilon}
\nc{\sigmat}{\widetilde\sigma}
\nc{\ostar}{\overline{*}}

\nc{\mchar}{\mrm{Char}}
\nc{\Hom}{\mrm{Hom}}
\nc{\id}{\mrm{id}}

\nc{\remark}{\noindent{\bf{Remark:}}}
\nc{\remarks}{\noindent{\bf{Remarks:}}}

 \nc{\delete}[1]{}
 \nc{\grad}[1]{^{({#1})}}
 \nc{\fil}[1]{_{#1}}

\nc{\BA}{{\Bbb A}} \nc{\CC}{{\Bbb C}} \nc{\DD}{{\Bbb D}}
\nc{\EE}{{\Bbb E}} \nc{\FF}{{\Bbb F}} \nc{\GG}{{\Bbb G}}
\nc{\HH}{{\Bbb H}} \nc{\LL}{{\Bbb L}} \nc{\NN}{{\Bbb N}}
\nc{\PP}{{\Bbb P}} \nc{\QQ}{{\Bbb Q}} \nc{\RR}{{\Bbb R}}
\nc{\TT}{{\Bbb T}} \nc{\VV}{{\Bbb V}} \nc{\ZZ}{{\Bbb Z}}
\nc{\Cal}[1]{{\mathcal {#1}}}
\nc{\mop}[1]{\mathop{\hbox {\rm #1} }\nolimits}
\nc{\smop}[1]{\mathop{\hbox {\eightrm #1} }\nolimits}
\nc{\mopl}[1]{\mathop{\hbox {\rm #1} }\limits}
\nc{\frakg}{{\frak g}}
\nc{\g}[1]{{\frak {#1}}}
\def \restr#1{\mathstrut_{\textstyle |}\raise-8pt\hbox{$\scriptstyle #1$}}
\def \srestr#1{\mathstrut_{\scriptstyle |}\hbox to
  -1.5pt{}\raise-4pt\hbox{$\scriptscriptstyle #1$}}
\nc{\wt}{\widetilde}
\nc{\wh}{\widehat}
\nc{\un}{\hbox{\bf 1}}
\nc{\redtext}[1]{\textcolor{red}{\tt #1}}
\nc{\bluetext}[1]{\textcolor{blue}{#1}}
\nc{\comment}[1]{[[{\tt {#1}}]] }
\nc{\R}{\mathbb R}
\nc\fleche[1]{\mathop{\hbox to #1 mm{\rightarrowfill}}\limits}
\def\semi{\mathrel{\times}\kern -.85pt\joinrel\mathrel{\raise
    1.4pt\hbox{${\scriptscriptstyle |}$}}}
\nc{\np}{/\hskip -2.3mm\pi}
\nc{\snp}{/\hskip -1.8mm\pi}

\def\ta1{{\scalebox{0.2}{ 
\begin{picture}(12,12)(38,-38)
\SetWidth{0.5} \SetColor{Black} \Vertex(45,-33){5.66}
\end{picture}}}}

\begin{document}


\title[A closed formula for the Magnus Expansion]
      {The Magnus expansion, trees and \\ Knuth's rotation correspondence}

\author{Kurusch Ebrahimi-Fard}
\address{Instituto de Ciencias Matem\'aticas,
		C/ Nicol\'as Cabrera, no.~13-15, 28049 Madrid, Spain.
		On leave from Univ. de Haute Alsace, Mulhouse, France}
         \email{kurusch@icmat.es, kurusch.ebrahimi-fard@uha.fr}         
         \urladdr{www.icmat.es/kurusch}

\author{Dominique Manchon}
\address{Univ.~Blaise Pascal,
         	C.N.R.S.-UMR 6620,
         	63177 Aubi\`ere, France}       
         \email{manchon@math.univ-bpclermont.fr}
         \urladdr{http://math.univ-bpclermont.fr/$\sim$manchon/}

\date{March 12th, 2012}

\begin{abstract}
W.~Magnus introduced a particular differential equation characterizing the logarithm of the solution of linear initial value problems for linear operators. The recursive solution of this differential equation leads to a peculiar Lie series, which is known as Magnus expansion, and involves Bernoulli numbers, iterated Lie brackets and integrals. This paper aims at obtaining further insights into the fine structure of the Magnus expansion. By using basic combinatorics on planar rooted trees we prove a closed formula for the Magnus expansion in the context of free dendriform algebra. From this, by using a well-known dendriform algebra structure on the vector space generated by the disjoint union of the symmetric groups, we derive the Mielnik--Pleba\'nski--Strichartz formula for the continuous Baker--Campbell--Hausdorff series.  
\end{abstract}

\maketitle

\begin{quote}
{\small{{\bf{key words}}: Magnus expansion; $B$-series; trees; pre-Lie algebra; dendriform algebra, Rota--Baxter algebra; permutations.}}
\end{quote}


\tableofcontents


\section{Introduction}
\label{sect:intro}

During the last decade, some surprising convergence of different areas of mathematical sciences has occurred. The seemingly separate fields: numerical integration methods, Lyons' rough path theory \cite{Gubinelli,Lyons}, Ecalle's mould calculus \cite{CHNT}, and Connes' noncommutative geometry \cite{CoMo}, share a common algebraic formalism where algebraic structures on trees and its underlying combinatorics are central.

Let us give two closely related examples coming from the theory of numerical integration. First we mention the pioneering work of J.~Butcher on an algebraic theory of integration methods in the 1960s and 1970s \cite{Butcher1,HWL}. Butcher's $B$-series can be seen as a generalization of Taylor series, in which rooted trees naturally appear. They uniquely represent elementary differentials, as Cayley noticed already in his classical 1857 paper \cite{Cayley}. Motivated by the problem of extending Butcher's work to the construction of generalized Runge--Kutta methods for integration of differential equations evolving on Lie groups, Munthe-Kaas introduced in \cite{MK1,MK2} the notion of Lie--Butcher series. In applications of both Butcher and Lie--Butcher series, algebraic and related combinatorial structures on rooted trees (in non-planar and planar version, respectively,) play an essential role. Since then, identifying genuine algebraic structures became a useful part of the theory of numerical integration methods, see e.g.~\cite{CHV,ChaMu,LunMun1,LunMun2,Murua}. The other example evolved out of W.~Magnus' seminal 1954 paper \cite{Magnus}, where the author introduced a particular differential equation characterizing the exponential solution of differential equations for a linear operator in terms of a Lie series. The latter is known as Magnus expansion and has become a well-known tool in the solution and approximation theory of linear initial value problems \cite{BCOR,MielPleb,Strichartz}. Iserles and collaborators \cite{Iserles1,Iserles2,Iserles3} were the first to use planar tree structures in an intriguing way to study the Magnus expansion in the context of numerical integration.\\

Recently it became clear that in general most of the combinatorial structures on trees can be traced back to the fact that free pre-Lie and dendriform algebras are naturally described in terms of rooted trees and planar binary trees, respectively \cite{ChaLiv,Loday,LR}. Indeed, trees provide genuine examples for combinatorial objects spanning connected graded locally finite-dimensional vector spaces, on which rich and various extra algebraic patterns are given in explicit terms by so-called tree grafting operations. These algebraic structures are summarized by the notion of combinatorial Hopf algebra, see e.g.~\cite{Hoffman}.

As it turns out, both Butcher's $B$-series and Magnus' expansion are most naturally described in terms of pre-Lie algebras. Indeed, Chapoton \cite{Chap1} was the first to study $B$-series as genuine formal series of rooted trees in the free pre-Lie algebra in one generator. In fact, later it became clear that also Magnus' series is contained, though in more disguised form, in his 2002 preprint. See also Murua's work \cite{Murua} for a link of Magnus' series to the Butcher--Connes--Kreimer Hopf algebra of rooted trees, and the underlying pre-Lie algebra. Based on the work by Iserles et al., we studied in \cite{EM1,EM2} the Magnus expansion in the light of its underlying pre-Lie structure, using the corresponding dendriform algebra. Later we realized that Agrachev and Gamkrelidze \cite{AG}  wrote down the Magnus expansion in the general context of chronological algebras, which is another name for pre-Lie algebras, as early as 1981. The interpretation as a logarithm, however, necessitates the dendriform structure as well.\\  

This paper is a continuation of our work in \cite{EM1}. It aims at obtaining further insights into the fine structure of the Magnus expansion by using basic combinatorial methods steaming from the description of free dendriform algebra in terms of planar binary trees. This allows us to present a closed formula for the Magnus expansion. A dendriform algebra structure on the linear span of the symmetric groups \cite{LR2} allows us to derive the continuous Baker--Campbell--Hausdorff series. In other words, we recover the known Mielnik--Pleba\'nski--Strichartz formula for the (classical) Magnus expansion. Using fairly elementary tools, our result may be seen as a pedestrian approach to parts of Chapoton's et al.~\cite{Chap2,ChaPat} and Thibon's et al.~\cite{gelfand} work on the Magnus series, where descent algebra, operads, and the theory of non-commutative symmetric functions play a dominant role.\\

The paper is organized as follows: in Section \ref{sect:AlgPri} we introduce the required structures, both combinatorial (planar binary trees, planar rooted trees and Knuth's rotation correspondence) and algebraic (Rota--Baxter, dendriform and pre-Lie algebras). We show how the dendriform algebra structure on the vector space spanned by planar binary trees together with the associative product is transported through Knuth's rotation correspondence to the vector space spanned by planar rooted trees. We describe the classical Magnus expansion in Section \ref{sect:MagExp}, writing it in terms of the pre-Lie product given on the space of locally integrable matrix-valued numerical functions by:
\begin{equation*}
	(f\rhd g)(s):=\left[\int_0^s f(u)\,du,\,g(s)\right].
\end{equation*}
The corresponding dendriform products are given by:
\begin{equation*}
	(f\succ g)(s):=\left(\int_0^s f(u)\,du\right)g(s),\hskip 8mm (f\prec g)(s):=f(s)\left(\int_0^s g(u)\, du\right).
\end{equation*}
In Section \ref{sect:LDEs} we look for a closed formula (Theorem \ref{thm:dendMagnus}) in the completion of the free unital dendriform algebra, giving the logarithm of the solution $X$ of the linear dendriform equation:
\begin{equation*}
	X=\un+a\prec X.
\end{equation*}
A $q$-analog of this result, relying on the abstract notion of \textsl{descents} for planar binary trees can be found e.g.~in a recent paper by Chapoton \cite{Chap2}. In our approach however, we refrained from using tools from the theory of non-commutative symmetric functions \cite{gelfand, DHT}, and derived the formula directly by invoking elementary combinatorial methods. Through the rotation correspondence the planar rooted tree picture reveals to be particularly useful here, since the number of descents of a planar binary tree simply corresponds to the number of leaves of its planar rooted image, excluding the leftmost leaf. The formula in the free setting yields of course a similar formula in any complete filtered dendriform algebra (Corollary \ref{cor2:dendMagnus}). In the last section we recover from Theorem \ref{thm:dendMagnus} and Corollary \ref{cor2:dendMagnus} the well-known Mielnik--Pleba\'nski--Strichartz formula \cite{MielPleb, Strichartz} for the classical Magnus expansion, also know as continuous Baker--Campbell--Hausdorff series. For that purpose we use a dendriform algebra structure on the direct sum $\bigoplus_{n\ge 1}k[S_n]$, where $S_n$ is the permutation group of $n$ letters \cite{Foissy,LR2}, and the fact that the notion of descent for a planar binary tree matches well with the well-known corresponding notion for a permutation.\\

\noindent
{\bf{Acknowledgements:}} We thank H.~Munthe-Kaas and A.~Lundervold for discussions and remarks. The first author is supported by a Ram\'on y Cajal research grant from the Spanish government. Both authors were supported by the CNRS (GDR Renormalisation).\\


\section{Algebraic and combinatorial preliminaries}
\label{sect:AlgPri}

Let $k$ be a field of characteristic zero (in our case it will always be either $\RR$ or $\mathbb{C}$).


\subsection{Trees}
\label{sect:trees}

Recall that a tree $t$ is a connected and simply connected graph made out of vertices and edges, the sets of which we denote by $V(t)$ and $E(t)$, respectively.


\subsubsection{Planar binary trees}
\label{ssect:pbt}

A \textsl{planar binary tree} is a finite oriented tree given an embedding in the plane, such that all vertices have exactly two incoming edges and one outgoing edge. An edge can be internal (connecting two vertices) or external (with one loose end). The external incoming edges are the leaves. The root edge is the unique edge not ending in a vertex. For any planar binary tree $t$, a partial order on the set of its vertices $V(t)$ is defined as follows: $u,v \in V(t)$, $u<v$ if and only if there is a path from the root of $t$ through $u$ up to $v$.
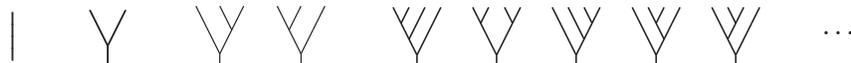
$$
	\raise -4pt\hbox{$\Big\vert$}
	\qquad\
	\scalebox{1.5}{\tree}
	\qquad\
	\treeA
	\quad
	\treeB
	\qquad\
	\treeE
	\quad
	\treeD
	\quad
	\treeC
	\quad
	\treeG	
	\quad 
	\treeF \qquad \ldots
$$
The single edge $\vert$ is the unique planar binary tree without internal vertices. We denote by $T^{bin}_{pl}$ (resp.~$\Cal T^{bin}_{pl}$) the set (resp.~the linear span) of planar binary trees. A simple grading for such trees is given in terms of the number of internal vertices. Above we listed all planar binary trees up to degree three. The number of trees of order $n$ is given by the Catalan number $c_n = \frac{(2n)!}{(n+1)!n!}$. The first ones are $1, 1, 2, 5, 14, 42, 132, \ldots$. Alternatively, one can use the number of leaves. Observe that for any pair of planar binary trees $t_1,t_2$ we can build up a new planar binary tree via the grafting operation, $t_3:=t_1 \vee t_2$, i.e.~by considering the unique ${\mathsf Y}$-shaped planar binary tree $\raise 3pt\hbox{$\scalebox{0.8}\tree$}$, and replacing the left branch (resp.~the right branch) by $t_1$ (resp.~$t_2$). 
$$
	\vert \vee \vert = \tree
	\quad\
	\tree \vee \vert = \treeB
	\quad\
	\vert \vee \tree = \treeA
	\quad\
	\tree \vee \tree = \treeD
	\quad\
	\vert \vee \treeB = \treeG.
$$

Any planar binary tree $t\not =\vert$ obviously expresses itself as $t_1\vee t_2$ in a unique way. The tree $t_1$ (resp. $t_2$) is the {\sl left part\/} (resp. the {\sl{right part}}) of $t$. The grafting operation $\vee$ makes $T_{pl}^{bin}$ the free magma algebra with one generator: the binary operation $\vee$ shows no relation of any kind, in particular it is neither commutative nor associative. Notice that this product is of degree one with respect to the grading in terms of internal vertices, i.e.~ for two trees $t_1,t_2$ of degrees $n_1,n_2$, respectively, the product $t_1 \vee t_2$ is of degree $n_1+n_2+1$. However, with respect to the leave number grading this product is of degree zero. We call the trees $\tau^{(n)}_r$, $\tau^{(n)}_l$ recursively defined by $\tau^{(0)}_{r}:= \vert=:\tau^{(0)}_{l}$ and $\tau^{(n+1)}_r:= | \vee \tau^{(n)}_r$, $\tau^{(n+1)}_l:= \tau^{(n)}_l \vee |$:
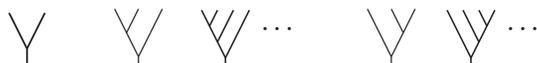
$$
\scalebox{1.5}{\tree}
	\qquad\
	\treeB
	\quad\
	\treeE\
	\cdots
	\qquad\
	\treeA
	\quad
	\treeC\
	\cdots
$$ 
right and left combs, respectively.


\subsubsection{Planar rooted trees}
\label{ssect:prt}

A {\it{planar rooted tree}} is a finite oriented tree given an embedding in the plane, such that all vertices, except one, the \textsl{root}, have arbitrarily many incoming edges and one outgoing edge. The root vertex has no outgoing edge, and the leaves have no incoming edges. 
$$
	\racine
	\qquad\
	\arbrea
	\qquad\
	\arbrebb
	\quad
	\arbreba
	\qquad\
	\arbreca
	\quad
	\arbrecc
	\quad
	\arbrecd
	\quad
	\arbrece	
	\quad
	\arbrecb \quad\ \cdots
$$

The single vertex $\racine$ is the unique rooted tree without edges. Note that we put the root at the bottom of the tree. The set (resp.~the linear span) of planar non-empty rooted trees will be denoted by $T_{pl}$ (resp.~$\Cal T_{pl}$). A natural grading for such trees is given in terms of the number of edges. Another one is given by the number of vertices. Observe that any rooted tree of degree bigger than zero writes in a unique way:
\begin{equation*}
	t=B_+(t_1\cdots t_n),
\end{equation*}
where $B_+$ associates to the forest $t_1\cdots t_n$ the planar tree obtained by grafting all the planar trees $t_j$, $j=1,\ldots,n$, on a common root. 
$$
	B_+(\racine)=\arbrea ,
	\quad\
	B_+(\racine\racine)=	\arbrebb ,
	\quad\
	B_+(\arbrea\racine)=\arbrecc ,
	\quad\
	B_+(\racine\arbrea)=\arbrece ,
	\quad\
	B_+(\racine\racine\racine)=\arbrecd.
$$
Sometimes, one finds the notation $t=[t_1\cdots t_n]$ in the literature \cite{Butcher1}. Note that the order in which the branch trees are displayed has to be taken into account. We introduce the \textsl{left Butcher product} of two planar rooted trees $t=B_+(t_1\cdots t_n)$ and $u=B_+(u_1\cdots u_p)$:
\begin{equation}
\label{butcherProd}
	t \butcher u := B_+(t u_1\cdots u_p).
\end{equation} 
Hence, it is defined by connecting the root of $t$ via a new edge to the root of $u$ such that $t$  becomes the leftmost branch tree. Observe that it is neither associative nor commutative. Moreover, it is clear that any rooted tree $t=B_+(t_1\cdots t_n) \in T_{pl}$ of degree bigger than zero uniquely decomposes as $t= t_1 \butcher t_2.$ The rooted trees recursively defined by $\ell^{(0)}:=\racine=:c^{(0)}$, and $\ell^{(n+1)}:=\ell^{(n)} \butcher \racine$, $c^{(n+1)}:=\racine \butcher c^{(n)}$ are called ladder trees and corollas, respectively.


\subsubsection{Knuth's correspondence between planar binary and planar rooted trees}
\label{sssect:pbt2prt}

A natural question is how to relate the two sets of planar trees just presented. Knuth described in \cite{Knuth68} a natural way to do this, known as rotation correspondence between planar binary and planar rooted trees. We only give a recursive description of this bijection denoted as $\Phi: T^{bin}_{pl} \to T_{pl}$, by defining $\Phi(\vert):=\racine$ and:
\begin{equation}
\label{phi}
	\Phi(t_1\vee t_2):=\Phi(t_1)\butcher\Phi(t_2).
\end{equation}
This map is well-defined and bijective, with its inverse recursively given by:
\begin{equation}
\label{phi-inv}
	\hspace{-2cm}\Phi^{-1}\big(B_+(t_1\cdots t_n)\big) 
	= \Phi^{-1}(t_1)  \vee \Phi^{-1}\big(B_+(t_2 \cdots t_n)\big) 
	= {\scalebox{0.8}{\treebplus}}
\end{equation}
The first few terms write:
$$
	\Phi(|)=\racine
	\qquad\
	\Phi(\!\tree)=\arbrea
	\qquad
	\Phi({\scalebox{0.8}{\treeA}})=\arbrebb
	\qquad\
	\Phi({\scalebox{0.8}{\treeB}})=\arbreba
$$
$$
	\Phi({\scalebox{0.8}{\treeE}})=\arbreca
	\qquad\
	\Phi({\scalebox{0.8}{\treeD}})=\arbrecc
	\qquad\
	\Phi({\scalebox{0.8}{\treeC}})=\arbrecd
	\qquad\
	\Phi({\scalebox{0.8}{\treeG}})=\arbrece	
	\qquad\ 
	\Phi({\scalebox{0.8}{\treeF}})=\arbrecb.
\vspace{0.2cm}
$$
Observe the compatibility with the gradings by the number of internal vertices in $T^{bin}_{pl}$ and the number of edges in  $T_{pl}$.
This simple bijection implies that the left Butcher product (\ref{butcherProd}) is also purely magmatic. Left and right combs $\tau^{(n)}_l$, $\tau^{(n)}_r$ map via $\Phi$ to the ladder trees $\ell^{(n)}$ and corollas $c^{(n)}$, respectively. For reasons to become clear in the sequel, we ask the reader to note the equality between the number of those leaves of a planar binary tree which point to the left, and the number of leaves of the corresponding planar rooted tree.\\


\subsection{Pre-Lie algebras}
\label{ssect:pL}

Recall that a left pre-Lie algebra $(A, \rhd)$ is a $k$-vector space $A$ equipped with an operation $\rhd: A \otimes A \rightarrow A$ subject to the following relation:
$$
	(a \rhd b) \rhd c - a \rhd ( b \rhd c) = (b \rhd a) \rhd c - b \rhd (a \rhd c).
$$
See e.g.~\cite{Manchon} for a survey on pre-Lie algebras. A genuine example is the pre-Lie algebra of vector fields. Let $M$ be a differentiable manifold equipped with a flat, torsion-free connection $\nabla$. The space of vector fields $\chi(M)$ can be given the structure of a pre-Lie algebra by defining the product $f \rhd g = \nabla_f g$. In the case of $M = \mathbb{R}^n$ with its canonical flat and torsion-free connection we have that for $f = \sum_{i=1}^n f_i \partial_i$ and $g=\sum_{j=1}^n g_j \partial_j$:
$$
	f \rhd g := \sum_{i=1}^n \left(\sum_{j=1}^n f_j\partial_jg_i\right)\partial_i.
$$
Recall from Chapoton and Livernet \cite{ChaLiv} that the basis of the free pre-Lie algebra in one generator, $\mathcal{P}(\!\racine)$, can be expressed in terms of undecorated, non-planar rooted trees. The set (resp.~the linear span) of the latter will be denoted by $T$ (resp.~$\Cal T$). In $T$ too, any rooted tree of degree bigger than zero writes $t=B_+(t_1\cdots t_n)$. However, the order in which the branch trees are displayed plays no role. The pre-Lie product in $\mathcal{P}(\!\racine)$ becomes very explicit in terms of tree grafting, that is, $t_1 \curvearrowright t_2$ is given by summing over all trees resulting from grafting successively the tree $t_1$ to each vertex of $t_2$: 
$$
	t_1  \curvearrowright t_2 := \sum_{v \in V(t_2)} t_1 \curvearrowright_{v} t_2,
$$
where $\curvearrowright_{v}$ denotes the grating of the root of $t_1$ via a new edge to vertex $v$ of $t_2$.
$$
	\racine  \curvearrowright \!\racine =  \arbrea\ ,\;\
	(\racine  \curvearrowright \!\racine) \curvearrowright \!\racine =  \arbreba\ ,\;\
	\racine  \curvearrowright (\!\racine \curvearrowright \!\racine) =  \arbreba +  \arbrebb\ ,\;\
	\racine  \curvearrowright (\!\racine  \curvearrowright (\!\racine \curvearrowright \!\racine)) =
	\arbreca
	+
	3 \arbrece
	+
	\arbrecd
	+
	\arbrecb.
$$
Any tree can of course be written as a polynomial expression in the generator $\racine$ using $\curvearrowright$ and suitable parenthesizing. See \cite{AG,Segal} for the description of monomial bases for free pre-Lie algebra.

\smallskip 

Let $f$ be any smooth vector field on $\R^n$. To go from the free pre-Lie algebra on one generator to the pre-Lie algebra of vector fields one applies the elementary differential map $\mathcal{F}_f: \mathcal{P}(\!\racine) \rightarrow \chi(\RR^n)$, which is defined as the unique pre-Lie algebra morphism such that  $\mathcal{F}_f[\!\racine] = f$.  Using standard notations, for $t=B_+(t_1\cdots t_n) \in T$ we have:
$$
	\mathcal{F}_f[t](x) = f^{(n)}(x)\big(\mathcal{F}_f[t_1](x), \dots, \mathcal{F}_f[t_n](x)\big).
$$ 
Elaborating on the introduction of this paper, we mention that Butcher's $B$-series can be seen as series expansions $b(\alpha) = \sum_{t\in T}\big(\alpha(t)/\sigma(t)\big)t$ in (the appropriate completion of) $\mathcal{P}(\!\racine)$. The linear function $\alpha$ on $\Cal P(\!\racine)$ maps trees to $k$, and $\sigma(t)$ is the symmetry factor of the non-planar tree $t$. Passing to usual $B$-series amounts to applying the elementary differential map $\mathcal{F}_{hf}$, where $hf$ is a smooth vector field on $\mathbb{R}^n$ multiplied by the so-called step size parameter $h$:
\begin{equation*}
	B_f(\alpha;y) = \sum_{t\in T}h^{V(t)}\frac{\alpha(t)}{\sigma(t)}\mathcal{F}_f[t](y). 
\end{equation*}
Indeed, a canonical $B$-series consists of a formal power series in the step size parameter $h$ containing elementary differentials and arbitrary coefficients encoded in a function $\alpha$ on the set of rooted trees $T$. Further below we will see another natural example of pre-Lie algebra in the context of numerical integration methods. 

\smallskip

As a final remark we mention that composition of $B$-series leads to what is called Butcher's group, which clarified important aspects in the theory of Runge--Kutta methods. Later, Chartier, Hairer and Vilmart introduced a so-called substitution law for $B$-series in the context of backward error analysis. See \cite{CHV,LunMun2} for concise reviews.\\


\subsection{Rota--Baxter algebras}
\label{sssect:dpse}
 
Recall that a Rota--Baxter algebra is a $k$-algebra $A$ endowed with a $k$-linear map $R: A \to A$ that satisfies the relation
\begin{equation}
\label{RB}
    R(a)R(b) = R\big(R(a)b + aR(b) + \theta ab\big),
\end{equation}
where $\theta \in k$. The map $R$ is called a {\sl Rota--Baxter operator of weight $\theta$\/}. The map $\widetilde{R}:=-\theta \textrm{id}-R$ also is a weight $\theta$ Rota--Baxter map. Both images $R(A)$ and $\tilde{R}(A)$ are subalgebras in $A$. One may think of (\ref{RB}) as a generalized integration by parts identity. Indeed, a simple example is given by the classical integration by parts rule showing that the ordinary Riemann integral is a weight zero Rota--Baxter map. Other examples can be found for instance in \cite{EM1,EM2}.  Observe that for an associative Rota--Baxter algebra the two compositions $a*b:=R(a)b+aR(b)+\theta ab$ and $a \rhd b := [R(a),b] - \theta ba$ define a new associative product and a pre-Lie product, respectively.\ignore{The pre-Lie algebra with product $\rhd$ is denoted by $P_R$.}\\


\subsection{Dendriform algebras}
\label{ssect:DendAlg}

We introduce the notion of dendriform algebra~\cite{Loday} over $k$, which is a $k$-vector space $D$ endowed with two bilinear operations $\prec$ and $\succ$ subject to the following three axioms:
\begin{eqnarray}
	(a\prec b)\prec c  &=& a\prec(b \prec c + b \succ c)        	\label{A1}\\
  	(a\succ b)\prec c  &=& a\succ(b\prec c)   				\label{A2}\\
   	a\succ(b\succ c)  &=& (a \prec b + a \succ b)\succ c        	\label{A3}.
\end{eqnarray}
In a commutative dendriform algebra (also known under the name \textsl{Zinbiel algebra}), the left and right operations are identified, that is, $x \succ y = y \prec x$. Axioms (\ref{A1})-(\ref{A3}) imply that for $a,b \in D$ the composition:
\begin{equation}
\label{dendassoc}
	a * b := a \prec b + a \succ b
\end{equation}	
defines an associative product. Hence, a dendriform algebra is an associative algebra together with a bimodule structure on itself, such that the associative product splits into the sum of the left- and right-module structures. Moreover, dendriform algebras are at the same time pre-Lie algebras. Indeed, one verifies that the two products:
\begin{equation}
\label{def:prelie}
    a \rhd b:= a\succ b-b\prec a,
    \hskip 12mm
    a \lhd b:= a\prec b-b\succ a
\end{equation}
are left pre-Lie and right pre-Lie, respectively. That is, we have:
\begin{eqnarray*}
    (a\rhd b)\rhd c-a\rhd(b\rhd c)&=& (b\rhd a)\rhd c-b\rhd(a\rhd c),\\
    (a\lhd b)\lhd c-a\lhd(b\lhd c)  &=& (a\lhd c)\lhd b-a\lhd(c\lhd b).
\end{eqnarray*}
These two pre-Lie products vanish if the dendriform algebra is commutative. Recall that a left pre-Lie algebra is Lie admissible \cite{AG,ChaLiv}, that is:
\begin{equation*}
    [a,b]:=a\lhd b-b\lhd a
\end{equation*}
defines a Lie bracket. An analogous statement holds for right pre-Lie algebras. Moreover, the Lie brackets following from the associative operation (\ref{dendassoc}) and the pre-Lie operations (\ref{def:prelie}) all define the same Lie bracket. For any dendriform algebra $A$ we denote by $\overline A = A \oplus k.\un$ the corresponding dendriform algebra augmented by a unit $\un$, with the following rules:
\begin{equation*}
    a \prec \un := a =: \un \succ a
    \hskip 12mm
    \un \prec a := 0 =: a \succ \un,
\end{equation*}
implying $a*\un=\un*a=a$. Note that the equality $\un*\un=\un$ makes sense, but that $\un \prec \un$ and $\un \succ \un$ are not defined. \\

Now suppose that the dendriform $A$ is complete with respect to the topology given by a decreasing filtration $A=A^1\supset A^2\supset A^3\supset\cdots$ compatible with the dendriform structure, in the sense that $A^p\prec A^q\subset A^{p+q}$ and $A^p\succ A^q\subset A^{p+q}$ for any $p,q\ge 1$. In the unital algebra we can then define the exponential and logarithm map in terms of the associative product~(\ref{dendassoc}): 
$$
	\exp^*(x):=\sum_{n \geq 0} x^{*n}/n!  
	\quad\ {\rm{resp.}} \quad\ 
	\log^*(\un+x):=-\sum_{n>0}(-1)^nx^{*n}/n. 
$$
Let $L_{a \succ} \left( b \right) := a \succ b =: R_{\succ b} \left( a \right) $. Note that $L_{a \succ} L_{b \succ} = L_{a \ast b \succ}$ and $R_{\prec a} R_{\prec b} = R_{\prec b \ast a}$. We recursively define the set of dendriform words in $\overline{A}$ for fixed elements $x_1,\ldots, x_n \in A$, $n \in \mathbb{N}$ by:
 \allowdisplaybreaks{
\begin{eqnarray*}
    w^{(0)}_{\prec}(x_1,\ldots, x_n) &:=& \un =:w^{(0)}_{\succ}(x_1,\ldots, x_n) \\
    w^{(n)}_{\prec}(x_1,\ldots, x_n) &:=& x_1 \prec \bigl(w^{(n-1)}_\prec(x_2,\ldots, x_n)\bigr)\\
    w^{(n)}_{\succ}(x_1,\ldots, x_n) &:=& \bigl(w^{(n-1)}_\succ(x_1,\ldots, x_{n-1})\bigr)\succ x_n.
\end{eqnarray*}}
In case that $x_1=\cdots = x_n=x$ we simply write $w^{(n)}_{\prec}(x,\ldots, x)=x^{(n)}_{\prec}$ and 
 $w^{(n)}_{\succ}(x,\ldots, x)= x^{(n)}_{\succ}$.\\

A simple example of dendriform algebra is given in by the Riemann integral on an algebra $\mathcal{F}$ of locally integrable functions over $\mathbb{R}$. The dendriform left and right products are:
\begin{equation*}
	(f \prec g)(t):=f(t) \left(\int_0^tg(s)ds\right)
	\qquad
	(f \succ g)(t):=\left(\int_0^tf(s)ds\right) g(t).
\end{equation*}
With $I(f)(t):=\int_0^t f(s)ds$, the dendriform axioms simply encode integration by parts
\begin{eqnarray*} 
(f \prec g) \prec h &=& f I (g) I (h) = f I\big( g I (h)\big) + f I\big(I (g) h\big) 
				 = f \prec (g \prec h) + f \prec (g \succ h) \\
(f \succ g) \prec h &=& \big( I(f) g\big) I (h) = I(f)  \big(g I (h)\big)
   				 = f \succ (g \prec h) \\
f \succ (g \succ h) &=& I(f)  I(g) h = I\big(f I (g)\big)h + I\big(I (f) g\big) h 
				  = (f \prec g) \succ h + (f \succ g) \succ h.
\end{eqnarray*}
The pre-Lie product $(f \rhd g)$ is given by the bracket $[I(f),g].$ The last example generalizes to any associative Rota--Baxter algebra $(\Cal A,R)$ of weight $\theta$, giving rise to a dendriform algebra $(\Cal A,\prec,\succ)$ defined in terms of:
 \allowdisplaybreaks{
\begin{eqnarray*}
    a \prec b &:=& aR(b)+\theta ab =-a\widetilde{R}(b),\hskip 8mm a \succ b:=R(a)b.
\end{eqnarray*}}
The dendriform associative and left pre-Lie products are explicitly given for $a,b \in \Cal A$ by:
\begin{eqnarray}
\label{RBdoublePreLie}
	a*b=R(a)b+aR(b)+\theta ab \quad\ {\rm{resp.}} \quad\ a \rhd b = [R(a),b] - \theta ba.
\end{eqnarray}


\subsubsection{Dendriform algebra structure on planar trees}
\label{sssect:DendAlgTree}

In \cite{Loday} it was shown that planar binary trees different from $\vert$ generate the free dendriform algebra in one generator. The associative product for two trees $s=s_{1}\vee s_{2}$ and  $t=t_{1}\vee t_{2}$ in ${\Cal T'}^{bin}_{pl}:={\Cal T}^{bin}_{pl}-\{\vert\}$ is given recursively by:
\begin{equation*}
	s * t = s_{1}\vee (s_{2} * t) + (s * t_{1}) \vee t_{2}.
\end{equation*}
The two terms on the right define the dendriform compositions $\prec$ and $\succ$ respectively. A simple computation shows the following link between the dendriform structure and the magmatic product:
\begin{equation}
\label{magdend}
	s\vee t=s\succ\raise 3pt\hbox{$\scalebox{0.8}\tree$}\prec t.
\end{equation}
The tree $\vert$ can be taken as the unit for the corresponding augmented dendriform algebra. For any dendriform algebra $A$ there is a unique morphism $F_a :{ \Cal T'}^{bin}_{pl} \to A$. Using \eqref{magdend}, it is recursively given by $F_a(\raise 3pt\hbox{$\scalebox{0.8}\tree$})=a$ and:
\begin{eqnarray}
	F_a(t)	&=&	F_a(t_1\vee t_2)\nonumber \\
			&=&	F_a(t_1\succ\raise 3pt\hbox{$\scalebox{0.8}\tree$}\prec t_2) \nonumber\\
			&=&	F_a(t_1)\succ a\prec F_a(t_2).\label{def:dendmap}
\end{eqnarray}
This unital dendriform algebra structure on $ \Cal T^{bin}_{pl}$ transfers via Knuth's correspondence between planar binary and planar rooted trees (\ref{phi}) to the dendriform associative product defined for $s=s_1\butcher s_2,\, t=t_1\butcher t_2 \in \Cal T_{pl}$ by:
\begin{equation*}
	s \star t = s_{1} \butcher (s_{2} \star t) + (s \star t_{1}) \butcher t_{2}.
\end{equation*}
The tree $\racine$ is the unit for this product. In the free dendriform algebra with one generator $a$, the dendriform words $a^{(n)}_{\prec}$ and $a^{(n)}_{\succ}$ translate into the left and right comb trees, respectively. And via Knuth's correspondence they are mapped to the ladder trees and corolla trees, respectively. Observe that for rooted ladder trees we find:
\begin{eqnarray*}
	\ell^{(n)} \star \ell^{(m)} 	
	&=&  (\ell^{(n-1)} \butcher \ell^{(0)}) \star (\ell^{(m-1)} \butcher \ell^{(0)})\\
	&=& (\ell^{(n-1)} \butcher \ell^{(m)}) + (\ell^{(n)} \star \ell^{(m-1)} )  \butcher \ell^{(0)}\\
	&=&\sum_{r= 0}^{m}\Big(\cdots\big((\ell^{(n-1)}\butcher\ell^{(m-r)})\butcher\underbrace{\ell^{(0)}\big)
				\butcher\cdots\Big)\butcher\ell^{(0)}}_{\smop{$\scriptstyle r$ times}}.
\end{eqnarray*}
For instance:
$$
	\arbrea \star \arbrea = \arbrebb + \arbreba 
	\qquad\
	\arbrea \star \arbreba = \arbrece + \arbrecb + \arbreca
	\qquad\
	 \arbreba \star \arbrea = \arbrecc + \arbreca.
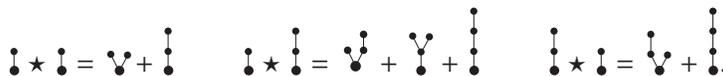
$$


\section{The Magnus expansion}
\label{sect:MagExp}

We start by recalling how to solve the linear initial value problem (IVP):  
$$
	\frac{d}{dt} X(t) = A(t)X(t),\quad X(0)=X_0,
$$ 
with $A(t)$ being, for instance, a matrix valued differentiable function. In fact, let us further simplify the problem by assuming scalar-valued functions, i.e.~we ignore any commutativity issues. Then the solution of the IVP in terms of the exponential map is $X(t)=\exp(\int_0^t A(x) dx)X_0$. Expanding the exponential and making use of the integration by parts rule immediately yields the identity:
\begin{equation}
\label{ex:classical}
	 1 + \int_0^tA(s_1) ds_1 \
	   + \int_0^t A(s_1)\int_0^{s_1}A(s_2) ds_2ds_1\
	   + \cdots
	   =\exp\Big(\int_0^t A(s) ds\Big).
\end{equation}
The left hand side, known as Dyson--Chen series, corresponds to the formal solution of the integral equation: 
$$
	X(t)=1+ \int_0^tA(s)X(s)ds,
$$
which is associated to the above IVP. This solution is still valid in the case where $A$ is a constant $n \times n$ matrix. However, this changes drastically in the general non-commutative case. Wilhelm Magnus described in a seminal 1954 paper \cite{Magnus} a particular differential equation for the matrix-valued function $\Omega(t;A)$: 
$$
	\dot{\Omega}(s;A) 	= \frac{ad_{\Omega(s;A)}}{e^{(ad_{\Omega(s;A)})} - 1}(A(s))
					= A(s) +  \sum_{n>0} \frac{B_n}{n!} ad_{\Omega(s;A)}^{(n)}(A(s)),
$$ 
such that the solution of the IVP writes:
$$
	X(t)=\exp\big(\Omega(t;A)\big)X_0.
$$
It is clear that $\Omega(0;A)=0$, hence $\Omega(t;A)=\int_0^t\dot\Omega(s;A)\,ds$. The $B_n$ are the Bernoulli numbers:
\begin{equation*}
	B_0=1,\ B_1=-\frac 12,\ B_2=\frac 16,\ B_4=-\frac{1}{30},\ldots \ {\rm{and}}\;\; B_{2k+1}=0 \hbox{ for }k\ge 1.
\end{equation*}
As usual, $ad^{(n)}_U(W)$ stands for the $n$-fold iterated Lie bracket $[U,[U,\cdots [U,W]]\cdots]$. Let us write down the first few terms of what is called Magnus' series, $\Omega(s;A)=\sum_{n \ge 0} \Omega_{n}(s;A)$, following from Picard iteration to solve the above recursion:
\begin{eqnarray*}
	\dot{\Omega}(s;\lambda A)&=&\lambda A(s) -  \frac{\lambda^2}{2}\left[\int_0^s\!\!\!\!A(x)dx ,A(s)\right] \\
		&&\qquad +\frac{\lambda^3 }{4}\left[\int_0^s\!\left[\int_0^y\!\!\!\!A(x)dx ,A(y)\right]dy,A(s)\right]
		+ \frac{\lambda^3}{12}\left[\int_0^s\!\!\!\!A(x)dx,\left[\int_0^s\!\!\!\!A(y)dy ,A(s)\right]\right] +\cdots,
\end{eqnarray*}
where we introduced a dummy parameter $\lambda$. Observe that the Lie bracket at order $\lambda^2$ simply serves to eliminate the second term on the righthand side:
$$
	\frac{1}{2}\Big(\int_0^t A(s) ds\Big)^2 = \frac{1}{2}\Big(\int_0^t \int_0^{s_1} A(s_1)A(s_2) ds_2ds_1 + 
												\int_0^t \int_0^{s_1} A(s_2) A(s_1)ds_2ds_1 \Big),
$$
such that up to second order in $\lambda$:
$$
	\exp\left( \int_0^t \lambda A(s) -  \frac{\lambda^2}{2}\int_0^t \left[\int_0^s\!\!\!\!A(x)dx ,A(s)\right] ds +\mathcal{O}(3)\right) 
	=  \lambda\!\! \int_0^t A(s) + \lambda^2\!\! \int_0^t \int_0^{s_1} A(s_1)A(s_2) ds_2ds_1 +\mathcal{O}(\lambda^3).
$$
Let us remark that since its publication, Magnus' paper received much attention and triggered important progress in both applied mathematics and physics. We refer the reader to \cite{BCOR} for  a comprehensive overview of its applications including a concise review of its background. Iserles et al. \cite{Iserles1,Iserles2,Iserles3} used planar tree to explore Magnus' expansion.  Magnus' result has also been explored from a more algebraic-combinatorial perspective using operads, pre-Lie algebras, dendriform algebras and noncommutative symmetric functions \cite{Chap1,Chap2,EM1,EM2,gelfand,Murua}. \\


\subsection{The pre-Lie Magnus expansion}
\label{ssect:pLMagExp}

Motivated by Spitzer's work \cite{Spitzer} in probability theory, Baxter~\cite{Baxter}, see also \cite{Atkinson}, followed a more general approach to the above IVP suggesting to look at the following fixpoint equation:
\begin{equation}
\label{IVP}
    X = 1 + \lambda R(aX),
\end{equation}
in $A[[\lambda]]$, where $A$ is a commutative unital algebra. The linear map $R$ on $A$ is supposed to satisfy the Rota--Baxter relation of scalar weight $\theta$ (\ref{RB}). The solution of (\ref{IVP}) is known as Spitzer's identity. In \cite{EM1} we generalized this to a noncommutative Rota--Baxter algebra $A$ by showing that equation (\ref{IVP}) is solved by $X(\lambda a)=\exp(R(\Omega'(\lambda a))$, where:
\begin{equation}
\label{plm-ag}
	\Omega'(\lambda a)=\lambda a + \sum_{n>0} \frac{B_n}{n!} L_{\Omega'(a) \rhd}^{(n)}\;(\lambda a)
\end{equation}
denotes the {\it{pre-Lie Magnus expansion}}. Here $L_{x\rhd}(y):=x \rhd y$ where $\rhd$ is the pre-Lie product defined in (\ref{RBdoublePreLie}). The prime notation shall remind the reader of Magnus' original differential equation. The first few terms are:
\begin{equation*}
	\Omega'(\lambda a)=\lambda  a - \lambda^2 \frac 12 a\rhd a 
					+ \lambda^3 \left(\frac 14 (a\rhd a)\rhd a +  \frac 1{12} a\rhd(a\rhd a)\right)+\cdots.
\end{equation*}
The reader is refereed to \cite{EM1,EM2} for more details. This leads to the following so-called noncommutative Spitzer identity:
\begin{equation}
\label{Spitzer}
	 1 + \lambda R(a)\
	    + \lambda^2R(aR(a))\
	    + \lambda^3R(aR(aR(a)))+ \cdots
	   =\exp\Big(R(\Omega'(\lambda a))\Big).
\end{equation}

Let us briefly note that proceeding analogously to the example of $B$-series in Paragraph \ref{ssect:pL}, we may abstract the pre-Lie Magnus expansion into (the appropriate completion of)  free pre-Lie algebra in one generator $\Cal P(\!\racine)$: 
$$
	\omega'(\beta)= \sum_{t\in T}\big(\beta(t)/\sigma(t)\big)t.
$$ 
Now, to go from the free pre-Lie algebra on one generator to the pre-Lie algebra $(\Cal A,\rhd)$ with the pre-Lie product defined in (\ref{RBdoublePreLie}), one applies for any $a\in\Cal A$ the map $\mathcal{F}_a: \mathcal{P}(\!\racine) \rightarrow \Cal A$, which is defined as the unique pre-Lie algebra morphism such that  $\mathcal{F}_a[\!\racine] = a$. For rooted trees of higher orders, we must recall that we can choose a monomial basis of $\mathcal{P}(\!\racine)$. Once the monomial basis is fixed, the decomposition of a tree $t \in \Cal T$ of order $n$ in this basis gives rise to a unique polynomial $p_{\curvearrowright}^t(\!\racine)$ of order $n$ in the generator $\racine$. Then $\mathcal{F}_a[t] = \mathcal{F}_a[p_{\curvearrowright}^t(\!\racine)] := p_\rhd^t(a)$, and the series $\omega'(\beta)$ in $\mathcal{P}(\!\racine)$ is mapped to $\Omega'(a)$. Of course, this picture demands for more details, such as for instance the way to chose the monomial basis. Moreover, we already remarked in \cite{EM1,EM2} that this approach will lead to a Magnus expansion with fewer terms than the classical series. A more complete account of this perspective on the Magnus expansion will be presented in a forthcoming work.


\subsection{Pre-Lie Magnus expansion and the group of formal flows}
\label{ssect:prelieMagnus}

Using the dendriform product $\prec$ given by $x\prec y=xR(y)$ (we put the weight $\theta$ to zero for simplicity), equation \eqref{Spitzer} in the last paragraph reads:
\begin{equation*}
	1+R\Big(\lambda a+\lambda^2 a\prec a +\lambda^3 a\prec(a\prec a)+\cdots\Big)
	=\exp\Big(R(\Omega'(\lambda a))\Big).
\end{equation*}
Using the associative product $*=\prec+\succ$ given by $x*y=R(x)y+xR(y)$ and extending the linear map $R$ to the associated unital dendriform algebra by setting $R(\un):=1$, yields:
\begin{equation*}
	X=\exp^*\big(\Omega'(\lambda a)\big),
\end{equation*}
where:
\begin{equation*} 
	X 	= \sum_{n\ge 0} (\lambda a)^{(n)}
		= \un + \lambda a + \lambda^2 a\prec a 
			+ \lambda^3 a \prec \left( a\prec a \right) 
			+ \lambda^4 a \prec \left( a \prec \left( a\prec a \right) \right) + \cdots
\end{equation*}
is the solution of the equation $X=\un+\lambda a \prec X$.

In \cite{EM1,EM2} we showed in the general dendriform algebra setting the following 
\begin{thm} \label{thm:pLMagnus}
{\rm{(\cite{EM1,EM2})}}
The element $\Omega' = \log^*(X(\lambda a))$ in $A$ satisfies the recursive formula similar to \eqref{plm-ag}:
\begin{equation*}
	\Omega' = \frac{L_{\Omega' \vartriangleright}}{e^{L_{\Omega' \vartriangleright} } - 1} (\lambda a) 
	   = \sum_{m \geqslant 0} \frac{B_m}{m!}  L_{\Omega' \rhd}^{(m)} (\lambda a) 
\end{equation*}
with $B_m$ the $m$-th Bernoulli number. The first few terms are:
\begin{equation*}
	\Omega'(\lambda a)=\lambda  a - \lambda^2 \frac 12 a\rhd a 
					+ \lambda^3 \left(\frac 14 (a\rhd a)\rhd a + \frac 1{12} a\rhd(a\rhd a)\right)+\cdots.
\end{equation*}\\
\end{thm}


\subsection{Group of formal flows}
\label{ssect:flow}

In \cite{EM3} we adapted a result on the Baker--Campbell--Hausdorff formula by Agrachev and Gamkrelidze \cite{AG} to dendriform algebras. Recall that a left pre-Lie algebra $A$ is called {\sl chronological algebra\/} in \cite{AG}, and that the left pre-Lie identity rewrites as:
\begin{equation*}
	L_{[a,b]\rhd}=[L_{a\rhd},L_{b\rhd}],
\end{equation*}
where as above $L_{a\rhd}:A\to A$ is defined by $L_{a\rhd}(b)=a\rhd b$, and where the bracket on the left-hand side is defined by $[a,b]:=a\rhd b-b\rhd a$. Recall that as a consequence this bracket satisfies the Jacobi identity. We denote by $A_{Lie}$ the Lie algebra with the aforementioned bracket. Suppose now that $A$ is a left pre-Lie algebra endowed with a decreasing filtration, namely $A=A_1\supset A_2\supset A_3\supset\cdots$, such that the intersection of the $A_j$'s reduces to $\{0\}$, and such that $A_p\rhd A_q\subset A_{p+q}$. Suppose moreover that $A$ is complete with respect to this filtration. The Baker--Campbell--Hausdorff formula:
\begin{equation*}
	C(a,b)=a+b+\frac 12[a,b]+\frac 1{12}([a,[a,b]]+[b,[b,a]])+\cdots
\end{equation*}
endows $A$ with a structure of pro-unipotent group. This group admits a more transparent presentation as follows. First, introduce a fictitious unit $\un$ such that $\un\rhd a=a\rhd\un=a$ for any $a\in A$, and define $W: A\to A$ by: 
$$
	W(a):=e^{L_{a\rhd}}\un-\un=a+\frac 12 a\rhd a+\frac 16 a\rhd(a\rhd a)+\cdots.
$$ 
The application $W$ is a bijection. The inverse is the pre-Lie Magnus expansion $\Omega'(a)$ introduced in Theorem \ref{thm:pLMagnus}. Transferring the BCH product by means of the map $W$, namely:
\begin{equation}
\label{diese}
	a\# b :=W\Big(C\big(\Omega(a),\Omega(b)\big)\Big),
\end{equation}
we have $W(a)\#W(b)=W\big(C(a,b)\big)=e^{L_{a\rhd}}e^{L_{b\rhd}}\un-\un$, hence $W(a)\#W(b)=W(a)+e^{L_{a\rhd}}W(b)$. The product $\#$ is thus given by the simple formula:
\begin{equation*}
	a\# b=a+e^{L_{\Omega(a)\rhd}}b.
\end{equation*}
The inverse is given by $a^{\#-1}=W\bigl(-\Omega(a)\big)=e^{-L_{\Omega(a)\rhd}}\un-\un$. In particular, when the pre-Lie product $\rhd$ is associative, this simplifies to $a\#b=a\rhd b+a+b$ and $a^{\#-1}=\frac 1{1+a}-1=\sum_{n\ge 1}(-1)^na_n$. If $(A,\rhd)$ and $(B,\rhd)$ are two such pre-Lie algebras and $\psi:A\to B$ is a filtration-preserving pre-Lie algebra morphism, it is immediate to check that
for any $a,b\in A$ we have:
\begin{equation*}
	\psi(a\# b)=\psi(a)\#\psi(b).
\end{equation*}
In other words, the group of formal flows is a functor from the category of complete filtered pre-Lie algebras to the category of groups. Concerning linear dendriform equations (\ref{eq:dendri-linear}) in a complete filtered unital dendriform algebra $A$ we observe the following interesting fact: for any collection $(a_1,\ldots,a_n)$ in $A$ with filtration degree $\ge 1$, the product  $X=X(a_1)*\cdots *X(a_n)$ where $X(a_i)$ is the solution of the equation $X(a_i)=\un+a_i\prec X(a_i)$ is the solution of the linear dendriform equation $X=\un+a\prec X$, with:
\begin{equation*}
	a=a_1\#\cdots\# a_n,
\end{equation*}
which also writes explicitly:
\begin{eqnarray*}
	a	&=& 	a_1 + \sum_{j=1}^{n-1}  e^{L_{\Omega(a_1 \# \cdots \# a_j)\rhd}}a_{j+1}\\
		&=&	a_1+e^{L_{\Omega(a_1)\rhd}}a_2+\cdots+e^{L_{\Omega(a_1)\rhd}}\cdots e^{L_{\Omega(a_{n-1})\rhd}}a_n.\\
\end{eqnarray*}


\section{Linear dendriform equations}
\label{sect:LDEs}

Let us emphasize that, although the series \eqref{plm-ag} has been known for a long time in the pure pre-Lie context \cite{AG}, the results in \cite{EM1,EM2}, interpreting this series as a logarithm, were obtained in the more restricted context of dendriform algebras. 

The starting point is the linear dendriform equation in $\Cal A[[\lambda]]$ below, where $\Cal A$ is now any unital dendriform algebra: 
\begin{equation}
\label{eq:dendri-linear} 
	X = \un + \lambda a \prec X
\end{equation}
for $a \in \Cal A$. Its formal solution is:
$$
	X= \sum_{n\ge 0} (\lambda a)^{(n)}_{\prec}= \un +  \lambda a 
		+ \lambda^2 a \prec a 
		+ \lambda^3 a \prec \left( a \prec a \right) 
		+ \lambda^4 a \prec \left( a \prec \left( a \prec a \right)\right) +\cdots. 
$$

\begin{rmk}{\rm{
In this work we deliberately have suppressed any Hopf algebra parlance. However, a few words are in order. Recall that the associative algebra $(\Cal T^{bin}_{pl},*)$ is the free algebra generated by the elements $\vert\vee T,\,T\in\Cal T^{bin}_{pl}$ \cite[Theorem 3.8]{LR}. The subalgebra generated by the right combs $\tau_r^{(n)}=\raise 3pt\hbox{$\scalebox{0.8}\tree$}_\prec^{(n)},\,n\ge 0$ is free and is a cocommutative Hopf subalgebra $\Cal H$ of $\Cal T^{bin}_{pl}$. The left combs $\tau_l^{(n)}=\raise 3pt\hbox{$\scalebox{0.8}\tree$}_\succ^{(n)},\,n\ge 0$ also belong to $\Cal H$. It follows that $\Cal H$ is isomorphic, as a Hopf algebra, to the Hopf algebra of noncommutative symmetric functions \cite{gelfand}. In Hopf algebraic terms the solution $X= \sum_{n\ge 0} (\lambda a)^{(n)}_{\prec}$ of (\ref{eq:dendri-linear}) is characterized as group-like, that is $\Delta(X)=X \otimes X$.}}\\
\end{rmk}


\subsection{A closed form for the logarithm}
\label{ssect:logOmega}

In this section, discarding the pre-Lie product, we give an explicit expression of $\log^*(X)$ in the free \textsl{dendriform} algebra in one generator\footnote{We recover the $q=1$ case of a $q$-analog formula by F.~Chapoton (see \cite[Proposition 5.10]{Chap2}).}. For this we use Knuth's rotation correspondence and consider the representation in terms of planar rooted trees rather than planar binary trees. The generator $a$ corresponds to the tree $\arbrea$. Then the solution $X$ is given by the sum of rooted ladder trees, $X=\un + \sum_{n > 0} \ell^{(n)} =: \un + L$. Here we identify $\un=\ell^{(0)}=\racine$. 
\begin{eqnarray*}
	\log^\star(X) = \log^\star(\un + L) &=& \sum_{n>0}\frac{(-1)^{n+1} L^{\star n}}{n}\\
	&=&\sum_{n >0}  \sum_{k>0} -\frac{(-1)^k}{k} 
					\sum_{i_1 + \cdots + i_k=n \atop i_j>0, j=1,\ldots,k } \ell^{(i_1)} \star \cdots \star \ell^{(i_k)}.
\end{eqnarray*}
We have omitted the parameter $\lambda$. Recall that the degree $|\ell^{(n)}|$ of the tree $\ell^{(n)}$ is equal to $n$, i.e.~its number of edges. The series above makes sense in the completion of the free dendriform algebra $\widehat {\Cal A}$ with respect to the grading. Recall that a ladder tree with $n$ edges stands for the dendriform word $a^{(n)}_{\prec}$, which in the classical example (\ref{ex:classical}) corresponds to the $n-1$ fold iterated integral $a(t)I(aI(a \cdots I(a)))(t)$.

\begin{rmk}{\rm{
Since $X$ is group-like, the logarithm is given by applying the eulerian idempotent, that is, $\log^\star(X)=\log^\star(\id)(X).$
}}\end{rmk}

\begin{thm} \label{thm:dendMagnus}
The element $\Omega' = \log^\star(X)$ in $\widehat {\Cal A}$ is given by the formula:
\begin{equation}
\label{eq:dendMagnus}
	\Omega' = \sum_{n > 0} \frac{1}{n} \sum_{\tau \in \Cal{T}_{pl} \atop |\tau|=n} \frac{(-1)^{\Cal{L}(\tau)-1}}{\binom{n-1}{\Cal{L}(\tau)-1}} \tau,
\end{equation}
where $\Cal{L}(\tau)$ denotes the number of leaves of the planar rooted tree $\tau$, and $|\tau|$ its degree, i.e.~its number of edges.
\end{thm}

The first few terms are:
\begin{equation*}
	\Omega'=  \arbrea - \frac 12 \arbrebb + \frac 12 \arbreba
					+ \frac 13 \arbreca 
					-  \frac 16 \arbrecc
					+ \frac 13 \arbrecd
					-  \frac 16 \arbrece
					-  \frac 16 \arbrecb
					+\cdots.
\end{equation*}

\begin{proof}
First observe that for any ordered composition $n=i_1+\cdots+i_k$ of some positive integer $n$, the expression $\ell^{(i_1)}\star\cdots\star\ell^{(i_k)}$ is a linear combination of planar rooted trees with $n$ edges, with coefficients equal to $0$ or $1$. Moreover, for any planar rooted tree $\tau$ with $n$ edges and $k$ leaves, there is a unique ordered composition $n=i_1+\cdots+i_k$ of the integer $n$ such that:

\begin{enumerate}
\item 
The corresponding $\star$-monomial of ladder trees $\ell^{(i_1)}\star\cdots\star\ell^{(i_k)}$ contains the tree $\tau$ precisely once.
\item
For any other composition $n=j_1+\cdots+j_r$ of the integer $n$, the $\star$-monomial $\ell^{(j_1)}\star\cdots\star\ell^{(j_r)}$ contains $\tau$ with a nontrivial coefficient only if the composition of the integer $n$ is finer than the first one.
\end{enumerate}

Recall that each leaf of a planar rooted tree is connected to the root by a unique shortest path, i.e.~including a minimal number of edges. The composition $(i_1,\ldots,i_k)$ is defined as follows: first we number the leaves consecutively from left to right. The rightmost leaf is linked down to the root by the path $\ell^{(i_k)}$ of length $i_k$ (i.e.~the height of the rightmost leaf). For any $s \in\{1,\ldots,k-1\}$, $i_{s}$ is the length of the path $\ell^{(i_s)}$ from leaf number $s$ to the vertex lying on the unique path joining leaf number $s+1$ down to the root. As an example the following six-leaved tree with 13 edges:
$$
	{\scalebox{0.8}{\colotree}}
$$
is associated to the ordered composition $13=2+1+1+3+2+4$ of its degree.

Next we write for fixed $n>0$:
$$
S_n(-1):=\sum_{k>0} -\frac{(-1)^k}{k} 
					\sum_{i_1 + \cdots + i_k=n \atop i_j>0, j=1,\ldots,k } \ell^{(i_1)} \star \cdots \star \ell^{(i_k)}
					=\sum_{\tau \in \Cal T_{pl} \atop |t|=n} s_n(\tau)\tau.
$$ 
The coefficients  in the last equality follow by projection, $s_n(\tau):=\langle Z_{\tau},S_n \rangle$, where $Z_{\tau}(\tau')=\delta_{\tau,\tau'}$. Now, we introduce a dummy parameter $\alpha$ in the sum:
$$
	S_n(\alpha):=\sum_{k>0} -\frac{\alpha^k}{k} 
					\sum_{i_1 + \cdots + i_k=n \atop i_j>0, j=1,\ldots,k } \ell^{(i_1)} \star \cdots \star \ell^{(i_k)},
$$
such that $\int_0^{-1} \frac{d}{d\alpha}S_n(\alpha)d\alpha = S_n(-1).$   Looking first at rooted ladder trees, we observe that in $S_n(\alpha)$ only one type of rooted ladder tree appears, i.e.~the one of length $n$, and:
$$
	\dot{S}_n(\alpha)=\sum_{k>0} - \alpha^{k-1} 
					\sum_{i_1 + \cdots + i_k=n \atop i_j>0, j=1,\ldots,k } \ell^{(i_1)} \star \cdots \star \ell^{(i_k)}
					= \sum_{k>0} - \alpha^{k-1} \binom{n-1}{k-1}\ell^{(n)} + \; {\small\textsl{{non-ladder\ rooted\ trees}}}.
$$
Here we have used that the total number of length $k$ ordered compositions of $n$ is $\binom{n-1}{k-1}$. Projecting onto the ladder part, and integrating, we find the number of rooted ladder trees $\ell^{(n)}$ in the sum $S_n(-1)$:
$$
	s_n(\ell^{(n)})	=\langle Z_{\ell^{(n)}},S_n(-1) \rangle = \int_{0}^{-1}  \sum_{k>0} -\alpha^{k-1} \binom{n-1}{k-1} d\alpha 
				= \int_{0}^{-1} -(1+\alpha)^{n-1} d\alpha = \frac{1}{n}.
$$ 
Now consider an arbitrary rooted tree $\tau$ of degree $n$ with $\Cal{L}(\tau)=k$ leaves. It is associated with an ordered composition $n=i_1+\cdots+i_k$ as explained above. Using the fact that considering an ordered composition of $n$ finer than the first one is nothing but choosing an ordered composition of $i_s$ for any $s=1,\ldots,k$, we get:
\begin{eqnarray*}
	s_n(\tau)=\langle Z_{\tau},S_n(-1) \rangle 
			&=& \int_{0}^{-1}-\alpha^{k-1}
				\sum_{j_1=1}^{i_1}\cdots \sum_{j_k=1}^{i_k} 
	 				\alpha^{j_1- 1}\binom{i_1-1}{j_1-1} \cdots  \alpha^{j_k - 1}\binom{i_k-1}{j_k-1} \,d\alpha\\
					&=& -\int_{0}^{-1}\alpha^{k-1}	 (1+\alpha)^{i_1+\cdots+i_k - k}d\alpha\\
					&=& (-1)^{k-1}\int_{0}^{1} \alpha^{k-1}	 (1-\alpha)^{n - k}\,d\alpha\\
					&=& (-1)^{k-1}\beta(k,n-k+1)\\
					&=& (-1)^{k-1} \frac{(k-1)!(n-k)!}{n!}\\
					&=& \frac{(-1)^{\Cal{L}(\tau)-1}}{n\binom{n-1}{\Cal{L}(\tau)-1}}.
\end{eqnarray*}
In the last step we used that $k$ equals the number $\Cal{L}(\tau)$ of leaves of tree $\tau$.
\end{proof}

Using the inverse of Knuth's correspondence we end up with an equivalent formulation in the planar binary tree picture. Recall that a leaf of a planar binary tree is a \textsl{descent} if it is not the leftmost one and if it is pointing to the left \cite{Chap2}. The rotation correspondence yields a bijection between the left-pointing leaves of a planar binary tree $t$ and the leaves of its image tree $\Phi(t)=\tau$. The leftmost leaf of $t$ (i.e.~the one which, by definition, is not a descent) is mapped to the leftmost leaf of $\tau$. As an immediate consequence we have:

\begin{cor} \label{cor:dendMagnus}
In the planar binary tree picture, the element $\Omega' = \log^*(X)$ in $\widehat {\Cal A}$ is given by the formula:
\begin{equation}
\label{eq:dendMagnus}
	\Omega' = \sum_{n > 0} \frac{1}{n} \sum_{t \in \Cal{T}_{pl}^{bin} \atop |t|=n} \frac{(-1)^{d(t)}}{\binom{n-1}{d(t)}} t,
\end{equation}
where $d(t)$ denotes the number of descents of the planar binary  tree $t$, and $|t|$ its degree, i.e.~its number of internal vertices.
\end{cor}

\begin{proof}
Follows from the rotation correspondence between planar binary trees and planar rooted trees. 
\end{proof}

Note that the extra sign $(-1)^{n-1}$ in Chapoton's formula \cite{Chap2} could be retrieved starting with the solution $Y=Y(a)$ of the equation $Y=\un+Y\succ a$ instead of $X$. Both are linked through the antipode provided the sign of the generator is changed, namely $Y(a)=S\big(X(-a)\big).$

\begin{cor} \label{cor2:dendMagnus}
For any complete filtered dendriform algebra $\Cal A=\Cal A_1\supset \Cal A_2\supset \Cal A_3\supset\cdots$ and for any $a\in \Cal A$, the element $\Omega'(a) = \log^*\big(X(a)\big)$ in $\Cal A$, where $X(a)$ is the solution of the linear dendriform equation $X(a)=\un+a\prec X(a)$, is given by the formula:
\begin{equation}
\label{eq:dendMagnus}
	\Omega' = \sum_{n > 0} \frac{1}{n} \sum_{t \in \Cal{T}_{pl}^{bin} \atop |t|=n} \frac{(-1)^{d(t)}}{\binom{n-1}{d(t)}} F_a(t),
\end{equation}
where $F_a: {{\Cal T}'}_{pl}^{bin}\to \Cal A$ is the unique dendriform algebra morphism defined in (\ref{def:dendmap}), such that $F_a(\raise 3pt\hbox{$\scalebox{0.8}\tree$})=a$.
\end{cor}
Hence, using the underlying dendriform algebra structure rather than the pre-Lie one, this result gives a closed formula for the Magnus expansion in Theorem \ref{thm:pLMagnus} in Paragraph \ref{ssect:prelieMagnus}.\\


\section{A formula by Mielnik--Pleba\'nski, and Strichartz}
\label{sect:PleMieSti}

We would like to give a more precise formulation of Theorem \ref{thm:dendMagnus}, when the dendriform algebra $\Cal A$ is an algebra of matrix-valued functions together with the Riemann integral as a Rota--Baxter operator (of weight $\theta=0$). In the next section we follow \cite{BrMePa,LR}. We will see that the continuous Baker--Campbell--Hausdorff series, also known as Mielnik--Pleba\'nski--Strichartz formula, gives a closed expression for the classical Magnus expansion due to deep structural reasons on the symmetric group algebra, i.e.~a natural underlying dendriform algebra, which matches with the dendriform algebra structure on planar binary trees.


\subsection{A dendriform structure on permutations}
\label{ssect:permdend}

$S_n$ is the group of permutations of $n$ elements, and $k[S_n]$ its group algebra. The concatenation of two permutations $\sigma \in S_n$ and $\tau \in S_m$ is the permutation $\sigma\times\tau \in S_{n+m}$ obtained by letting $\sigma$ act on the first $n$ elements and letting $\tau$ act on the $m$ last elements. An associative product is given by:
\begin{equation*}
	\sigma*\tau = \sum_{\omega \in \smop{Sh}_{n,m}}\omega\circ({\sigma\times\tau}).
\end{equation*}
$\mop{Sh}_{n,m}\subset S_{n+m}$ stands for the set of $(n,m)$ shuffles. We denote the graded connected algebra with this product by $\Cal H:=\bigoplus_{n\ge 0}\Cal H_n$,  where $\Cal H_n=k[S_n]$.

The crucial observation is that the product $*$ on permutations splits into $*=\prec+\succ$, where $\prec$ and $\succ$ are defined by:
\begin{equation*}
	\sigma\prec\tau=\sum_{\omega\in\smop {Sh}^2_{n,m}}\omega\circ({\sigma\times\tau}),
	\hskip 8mm
	\sigma\succ\tau=\sum_{\omega\in\smop {Sh}^1_{n,m}}\omega\circ({\sigma\times\tau}),
\end{equation*}
where $\mop{Sh}^1_{n,m}$, respectively $\mop{Sh}^2_{n,m}$ stands for the shuffles $\omega$ such that $\omega(n+m)=n+m$, resp.~$\omega(n)=n+m$. It is shown in \cite{LR2} that $\prec$ and $\succ$ endow the augmentation ideal $\Cal H':=\bigoplus_{n\ge 1}\Cal H_n$ with the structure of a dendriform algebra. 

\begin{remark}{\rm{
The (dendriform) algebra just described represents only half of the picture. Completing the structure amounts to what is known as Malvenuto--Reutenauer Hopf algebra \cite{MR}, which is a graded connected Hopf algebra on $\Cal H$. We refrain from giving more details and refer the reader to the references \cite{BrMePa,DHT,LR}. However, let us add that the coproduct splits also into two parts, thus endowing the Hopf algebra $\Cal H$ with the much richer structure of \textsl{bidendriform Hopf algebra}, see Foissy's work \cite{Foissy}.}}\\
\end{remark}


\subsection{Planar binary trees and permutations}
\label{ssect:pbtPerm}

The material presented in this paragraph is mostly borrowed from \cite{LR}. See also \cite{BrMePa} for a very detailed similar description. A bijective correspondence between permutations and \textsl{planar binary trees with levels} is well-known in combinatorics \cite{Han81}. A planar binary tree with levels is a planar binary tree $t$ with, say, $n$ internal vertices together with a bijective decreasing map $\varphi$ from the poset of its internal vertices into $\{1,\ldots,n\}$. Such a tree admits a graphical realization by drawing the internal vertices at the prescribed levels, with level $1$ being the top one and level $n$ being the deepest one. Any planar binary tree with levels $(t,\varphi)$ gives rise to two such trees $(t_1,\varphi_1)$ and $(t_2,\varphi_2)$, where $t=t_1\vee t_2$ and $\varphi_i$ is the "standardized" restriction of the injection $\varphi$ to the internal vertices of $t_i$, namely its composition on the left with the unique increasing bijection from its image onto $\{1,\ldots,|t_i|\}$, $i=1, 2$.

To any such tree $(t,\varphi)$ we can associate a permutation $\sigma_{t,\varphi}$ as follows: $\sigma_{t,\varphi}(i)$ is the level of the internal vertex $u_i$ situated between leaves $l_i$ and $l_{i+1}$ (the leftmost being the first and the rightmost being number $n+1$).  This correspondence $P$ is a bijection, the inverse of which is recursively given as follows: the permutation $\sigma\in S_n$ gives rise to two sequences of integers: the sequence before n and the sequence after n in $(\sigma_1,\ldots,\sigma_n)$. One of them may be empty. By ``standardizing" the integers in each sequence, they form a permutation. For instance $(341625)$ gives the two sequences $(341)$ and $(25)$, which, after standardizing, give $(231)$ and $(12)$. By induction these two permutations give rise to two trees with levels. The grafting $\vee$ of the two trees (in the order given above) gives the underlying tree of $P^{-1}(\sigma)$, and the original permutation is used to determine the levels of each vertex, namely $\varphi(u_j)=\sigma^{-1}(j)$.

Recall that the descent set of a permutation $\sigma\in S_n$ is the subset $D(\sigma)\subset\{1,\ldots,n-1\}$ of indices $i$ such that $\sigma(i)>\sigma(i+1)$. We denote by $d(\sigma)$ the cardinality of the descent set. The following lemma, the proof of which is simple, establishes a link with descents of planar binary trees:

\begin{lem}\label{lem:descents}
For any planar binary tree with levels $(t,\varphi)$, the correspondence
\begin{equation*}
	j\mapsto l_{j+1}:\{1,\ldots,n\}\to\{\mop{leaves of t}\}
\end{equation*}
restricts to a bijection from the descent set $D(\sigma_{t,\varphi})$ onto the descent set of $t$.
\end{lem}

Forgetting the levels, the bijection $P^{-1}$ provides a surjective map $\psi_n: S_n\to\left(T_{pl}^{bin}\right)_n$ for any $n\ge 0$. As a corollary of Lemma \ref{lem:descents}, we have:
\begin{equation}
\label{eq:descent}
	d(\sigma)=d\big(\psi(\sigma)\big)
\end{equation}
for any $\sigma\in S_n$. Dually, we get a linear injection $\psi_n^*:\left(\Cal T_{pl}^{bin}\right)_n\to k[S_n]$ given by:
\begin{equation*}
	\psi_n^*(t)=\sum_{\psi_n(\sigma)=t} \sigma.
\end{equation*}
These maps together give a degree zero linear injection $\psi^*$ from $\Cal T_{pl}^{bin}$ to the dendriform algebra $\Cal H$. The following theorem is due to Loday and Ronco (\cite[Theorem 3.1]{LR}, \cite[Proposition 5.3]{LR2}):

\begin{thm}\cite{LR2}\label{thm:dendmr}
The linear injection $\psi^*$ defined above is a unital dendriform algebra morphism.
\end{thm}

\begin{rmk}{\rm{
The map $\psi^*$ is even a bidendriform Hopf algebra morphism \cite{Foissy}.}}\\
\end{rmk}


\subsection{The Mielnik--Pleba\'nski--Strichartz formula}
\label{ssect:MPSformula}

Let $\Cal A$ be an algebra of locally integrable functions of one real variable, with values in some topological not necessarily commutative algebra, for example $\Cal M_N(\CC)$, $N\in \mathbb{N}$. The weight zero Rota-Baxter map $R:\Cal A\to \Cal A$ given by $Ra(s):=\int_0^s a(u)du$ endows $\Cal A$ with a dendriform algebra structure, as explained in Paragraph \ref{ssect:DendAlg} above. Fix some $a$ in $\Cal A$, suppose $s>0$, and define $\wt F_\sigma(a)\in \Cal A$ for any permutation $\sigma\in S_n,\,n\ge 1$ as follows:
\begin{equation*}
	\wt F_a(\sigma)(s):=\frac{d}{ds}\idotsint\limits_{0<u_n<\cdots <u_1<s}a(u_{\sigma_1})\cdots a(u_{\sigma_n})\,du_1\cdots du_n.
\end{equation*}
In particular, $\wt F_a\big((1)\big)=a$, $\wt F_a\big((12)\big)=aR(a)$, $\wt F_a\big((21)\big)=R(a)a$, $\wt F_a\big((123)\big)=aR(aR(a))$ and $\wt F_a\big((321)\big)=R(R(a)a)a$. Observe that there are permutations $\sigma \in S_n$, $n>2$, which, strictly speaking, can not be written as iterated operators, e.g.~$(231)$ and $(132)$. However, they can be written as iterated integrals in the above sense. Therefore, the following only applies to the special case of the dendriform algebra structure defined in terms of $Ra(s):=\int_0^s a(u)du$ on $\Cal A$. The crucial point is to transfer the dendriform calculations to the simplicies forming the integration domains. The correspondence $\wt F_a$ obviously extends uniquely to a linear map from the augmentation ideal of $\Cal H$ to $\Cal A$.\\

Before stating the theorem, let us introduce some notations: for any positive integer $N$ we denote by $\Delta_N^s\subset [0,s]^N$ the simplex $\{u=(u_1,\ldots,u_N),\,0<u_N<\cdots <u_1<s\}$. The symmetric group $S_N$ acts on $[0,s]^N$ by permutation of the variables, namely:
\begin{equation*}
	\sigma.(u_1,\ldots, u_N):=(u_{\sigma_1^{-1}},\ldots,u_{\sigma_N^{-1}}).
\end{equation*}
We have then:
\begin{equation*}
	\sigma.\Delta_N^s=\{u=(u_1,\ldots,u_N),\,0<u_{\sigma_N}<\cdots <u_{\sigma_1}<s\}.
\end{equation*}

\begin{thm}
\label{thm:main}
$\wt F_a$ is a dendriform algebra morphism from the augmentation ideal of $\Cal H$ to $\Cal A$, linked to the dendriform algebra morphism $F_a$ of Corollary \ref{cor2:dendMagnus} by:
\begin{equation*}
	F_a=\wt F_a\circ\psi^*.
\end{equation*}
\end{thm}

\begin{proof}
By direct computation: take $\sigma\in S_n$ and $\tau\in S_m$. Then,
\allowdisplaybreaks{
\begin{eqnarray*}
	R\big(\wt F_a(\sigma\prec\tau)\big)(s)
	&=&\sum_{\omega\in\smop{Sh}^2(n,m)}R\big(\wt F_a(\omega\circ(\sigma\times\tau)\big)(s)\\
	&=&\sum_{\omega\in\smop{Sh}^2(n,m)}\hskip 5mm
			\idotsint_{0<u_{\omega^{-1}_{n+m}}<\cdots <u_{\omega^{-1}_1}<s}
				a(u_{\sigma_1})\cdots a(u_{\sigma_n})a(u_{n+\tau_1})\cdots a(u_{n+\tau_m})\,du\\
	&=&\sum_{\omega\in\smop{Sh}^2(n,m)}\hskip 5mm
			\idotsint_{\omega^{-1}.\Delta_{n+m}}a(u_{\sigma_1})\cdots a(u_{\sigma_n})a(u_{n+\tau_1})\cdots a(u_{n+\tau_m})\,du\\
	&=&\idotsint_{\textstyle\bigcup\limits_{\omega\in\smop{Sh}^2(n,m)}\!\!\!\scriptstyle\omega^{-1}.\Delta_{n+m}}
				a(u_{\sigma_1})\cdots a(u_{\sigma_n})a(u_{n+\tau_1})\cdots a(u_{n+\tau_m})\,du\\
	&=&\idotsint_{{\scriptstyle 0<u_n<\cdots<u_1<s\atop\scriptstyle  0<v_m<\cdots <v_1<s},\,v_1<u_1}
				a(u_{\sigma_1})\cdots a(u_{\sigma_n})a(v_{\tau_1})\cdots a(v_{\tau_m})\,dudv\\
	&=&R\Big(\wt F_a(\sigma)R\big(\wt F_a(\tau)\big)\Big)\\
	&=&R\big(\wt F_a(\sigma)\prec \wt F_a(\tau)\big),
\end{eqnarray*}}
hence $\wt F_a(\sigma\prec\tau)=\wt F_a(\sigma)\prec \wt F_a(\tau)$. The corresponding computation for $\succ$ is entirely similar and left to the reader, using $\mop{Sh}^1(n,m)$ instead of $\mop{Sh}^2(n,m)$. The second statement is an immediate consequence of Theorem \ref{thm:dendmr} and the freeness of the dendriform algebra of planar binary trees.
\end{proof}

\begin{cor}\label{cor:mps}
The element $\Omega'(a) = \log^*\big(X(a)\big)$ in $\Cal A$, where $X(a)$ is the solution of the linear dendriform equation $X(a)=\un+a\prec X(a)$, is formally given by the series:
\begin{equation}
\label{eq:mps1}
	\Omega' = \sum_{n > 0} \frac{1}{n} \sum_{\sigma\in S_n} \frac{(-1)^{d(\sigma)}}{\binom{n-1}{d(\sigma)}} \wt F_a(\sigma).
\end{equation}
\end{cor}

\begin{proof}
This is an immediate consequence of Corollary \ref{cor2:dendMagnus}, Theorem \ref{thm:main} and Eq \eqref{eq:descent}.
\end{proof}

\begin{cor}[Mielnik--Pleba\'nski--Strichartz formula \cite{MielPleb,Strichartz}]\label{cor:mps2}
The Magnus element $\Omega(a) = R\big(\Omega'(a)\big)$ is formally given by the series:
\begin{equation}
\label{eq:mps2}
	\Omega(a)(s)= \sum_{n > 0}  \sum_{\sigma\in S_n} \frac{(-1)^{d(\sigma)}}{n \binom{n-1}{d(\sigma)}} 
	\idotsint\limits_{0<u_n<\cdots <u_1<s}a(u_{\sigma_1})\cdots a(u_{\sigma_n})\,du_1\cdots du_n.
\end{equation}
\end{cor}
\noindent
\textbf{Remark :} As $\Omega(a)$ is a Lie element, we can use the Dynkin-Specht-Wever theorem, so that we recover the formula in its original setting:
\begin{equation*}
	\Omega(a)(s)= \sum_{n > 0}  \sum_{\sigma\in S_n} \frac{(-1)^{d(\sigma)}}{n^2 \binom{n-1}{d(\sigma)}} 
		\idotsint\limits_{0<u_n<\cdots <u_1<s}
	[a(u_{\sigma_1}),[a(u_{\sigma_2}),\ldots [a(u_{\sigma_{n-1}},\,a(u_{\sigma_n})]\cdots]]\,du_1\cdots du_n.
\end{equation*}\\
The order of the $u_j$'s is reversed compared to the original, since $Z=\exp({\Omega(a)})$ solves the initial value problem $\dot Z=aZ$ rather than $\dot Z=Za$ as in \cite{Strichartz}.
\bigskip


\begin{thebibliography}{abcdsfgh}

\bibitem{AG}
	 A.~Agrachev, R.~Gamkrelidze, 
	 {\textit{Chronological algebras and nonstationary vector fields\/}}, 
	 J.~Sov.~Math.~{\bf{17}}, 1650--1675 (1981).
	 
\bibitem{AG2}A. A. Agrachev, R. V. Gamkrelidze, \textsl{The shuffle product and symmetric groups}, in K. D. Elworthy, W. N. Everitt, and E. B. Lee, editors, \textsl{Differential Equations,
Dynamical Systems and Control Science}, Lecture Notes in Pure and
Appl. Math. \textbf{152}, 365Ð82. Marcel Dekker, Inc., New York, 1994.

\bibitem{Atkinson}
    	F.~V.~Atkinson,
    	{\textsl{Some aspects of Baxter's functional equation}},
    	J.~ Math.~Anal.~Appl.~{\bf{7}}, 1--30 (1963).

\bibitem{Baxter}
    	G.~Baxter,
    	{\textsl{An analytic problem whose solution follows from a simple algebraic identity}},
    	Pacific J.~Math.~{\bf{10}}, 731--742 (1960).

\bibitem{BCOR}
    	S.~Blanes, F.~Casas, J.A.~Oteo, J.~Ros,
    	{\textsl{Magnus expansion: mathematical study and physical applications}},
    	Phys. Rep. {\bf{470}}, 151--238 (2009).

\bibitem{BrMePa}
	Ch.~Brouder, \^{A}.~Mestre, F.~Patras,
	{\sl{Tree expansion in time-dependent perturbation theory}},
	Journal of Mathematical Physics {\bf{51}} no.~7, 072104 (2010).

\bibitem{Butcher1}
	J.~C.~Butcher,
	{\textsl{An algebraic theory of integration methods}}, 
	 Math.~Comp.~{\bf{26}}, 79--106 (1972).

\bibitem{Cayley}
	A.~Cayley,
	{\textsl{On the Theory of Analytical Forms called Trees}}, 
	Philosophical Magazine {\bf{13}} 172--176 (1857). 
	
\bibitem{Chap1}
	F.~Chapoton, 
	{\textsl{Rooted trees and an exponential-like series}}, 
	{\texttt{arXiv:math/0209104}}.

\bibitem{Chap2}
	F.~Chapoton, 
	{\textsl{A rooted-trees q-series lifting a one-parameter family of Lie idempotents}},
	Algebra \& Number Theory {\bf{3}}, 611--636 (2009).

\bibitem{ChaPat}
	F.~Chapoton, F.~Patras,
	{\textsl{Enveloping algebras of preLie algebras, Solomon idempotents and the Magnus formula}},
	{\texttt{arXiv:1201.2159v1 [math.QA]}}.
	
\bibitem{CHNT}
	F.~Chapoton, F.~Hivert, J.C.~Novelli, J.Y.~Thibon,
	{\sl{An operational calculus for the mould operad}}, 
	Int. Math. Res. Not. {\bf{2008}}, no.~9, (2008). 

\bibitem{ChaLiv}
    	F.~Chapoton, M.~Livernet,
    	{\textsl{Pre-Lie algebras and the rooted trees operad}},
    	Int. Math. Res. Not. {\bf{2001}}, 395--408 (2001).
	
\bibitem{CHV}
	Ph.~Chartier, E.~Hairer, G.~Vilmart, 
	{\textsl{Algebraic structures of B-series}},
	Found.~Comput.~Math.~{\bf{10}} 407--427 (2010).
	
\bibitem{ChaMu}
	Ph.~Chartier, A.~Murua, 
	{\textsl{An algebraic theory of order}}, 
	M2AN {\bf{43}} 607--630 (2009). 

\bibitem{CoMo}
	A.~Connes, H.~Moscovici,
	{\sl{Hopf Algebras, Cyclic Cohomology and the Transverse Index Theorem}},
	Commun. Math. Phys. {\bf{198}}, 198--246 (1998).
	
\bibitem{DHT}
	G.~Duchamp, F.~Hivert, J-Y.~Thibon, 
	\textsl{Noncommutative symmetric functions VI: free quasi-symmetric functions  and related algebras},  
	J. Alg. Comput.~\textbf{12}, 671--717 (2002).
	
\bibitem{EM1}
	K.~Ebrahimi-Fard, D.~Manchon,
	{\textsl{A Magnus- and Fer-type formula in dendriform algebras\/}}, 
	Found. Comput.~Math.~{\bf{9}}, 295--316 (2009). 

\bibitem{EM2}
	K.~Ebrahimi-Fard, D.~Manchon,
	{\textsl{Dendriform Equations}}, 
	J.~Algebra {\bf{322}}, 4053--4079 (2009).

\bibitem{EM3}
	K.~Ebrahimi-Fard, D.~Manchon,
	{\textit{Twisted dendriform algebras and the pre-Lie Magnus expansion}}, 
	Journal of Pure and Applied Algebra {\bf{215}}, 2615--2627 (2011).	
	
\bibitem{Foissy}
	L.~Foissy, 
	\textsl{Bidendriform bialgebras, trees, and free quasi-symmetric functions}, 
	J.~Pure Appl.~Algebra \textbf{209} no.~2, 439--459 (2007).
	
\bibitem{gelfand}
    	I.~M.~Gelfand, D.~Krob, A.~Lascoux, B.~Leclerc, V.~Retakh, J.-Y.~Thibon,
    	{\textsl{Noncommutative symmetric functions}},
    	Adv. Math. {\bf{112}}, 218--348  (1995).
	
\bibitem{Gubinelli}
	M.~Gubinelli, 
	{\sl{Abstract integration, combinatorics of trees and differential equations}}, 
	in ``Combinatorics and Physics", Contemp.~Math.~{\bf{539}}, AMS, 135--151 (2011).

\bibitem{HWL}
	E.~Hairer, C.~Lubich, G.~Wanner, 
	{\textsl{Geometric numerical integration}} Structure-preserving algorithms for ordinary differential equations,
	vol. {\bf{31}}, Springer Series in Computational Mathematics. Springer-Verlag, Berlin, 2002.
	
\bibitem{Han81}
	P.~Hanlon, 
	\textsl{The fixed-point partition lattices}, 
	Pacific J.~Math.~{\bf{96}} no.~2, 319--341(1981). 

\bibitem{Hoffman}
	M.~Hoffman,
	{\textsl{Combinatorics of Rooted Trees and Hopf Algebras}}
	Trans.~Amer.~Math.~Soc.~{\bf{355}}, 3795--3811 (2003). 

\bibitem{Iserles1}
    	A.~Iserles, S.~P.~N{\o}rsett,
    	{\textsl{On the Solution of Linear Differential Equations in Lie Groups }}
    	Philosophical Transactions of the Royal Society A {\bf{357}}, 983--1020 (1999).
	
\bibitem{Iserles2}
    	A.~Iserles, H.~Z.~Munthe-Kaas, S.~P.~N{\o}rsett, A.~Zanna,
    	{\textsl{Lie-group methods}},
    	Acta Numerica {\bf{9}}, 215--365 (2000).

\bibitem{Iserles3}
    	A.~Iserles,
    	{\textsl{Expansions that grow on trees}},
    	Notices of the AMS {\bf{49}}, 430--440  (2002).

\bibitem{Knuth68}
	D.~E.~Knuth, 
	{\textsl{The art of computer programming I. Fundamental algorithms}},
	Addison-Wesley (1968).

\bibitem{Loday}
    	J.-L.~Loday,
    	{\textsl{Dialgebras}},
    	Lect.~Notes~Math.~1763, 7--66 (2001).
	
\bibitem{LR}	
	J.-L.~Loday, M.~Ronco, 
	\textsl{Hopf algebra of the planar binary trees}, 
	Adv.~Math.~\textbf{139}, 293--309 (1998).
	
\bibitem{LR2}
	J.-L.~Loday, M.~Ronco, 
	\textsl{Order structure and the algebra of permutations and of planar binary trees}, 
	J.~Alg.~Comb.~\textbf{15} no.~3, 253--270 (2002).

\bibitem{LunMun1}
	A.~Lundervold, H.~Munthe-Kaas,
	{\sl{Hopf algebras of formal diffeomorphisms and numerical integration on manifolds}},
	Contemporary Mathematics {\bf{539}}, 295--324 (2011). 

\bibitem{LunMun2}
	A.~Lundervold, H.~Munthe-Kaas,
	{\sl{On algebraic structures of numerical integration on vector spaces and manifolds}},
	{\texttt{arXiv:1112.4465v1 [math.NA]}}.

\bibitem{Lyons}
	T.~Lyons,
	{\textsl{Differential equations driven by rough paths}}, 
	Rev.~Mat.~Iberoamericana {\bf{14}}, 215--310 (1998). 

\bibitem{Magnus}
    	W.~Magnus,
    	{\textsl{On the exponential solution of differential equations for a linear operator}},
    	Commun. Pure Appl.~Math.~{\bf{7}}, 649--673 (1954).
	
\bibitem{MR}
	C.~Malvenuto, C.~Reutenauer, 
	\textsl{Duality between quasi-symmetric functions and the Solomon descent algebra}, 
	J.~Algebra \textbf{177} no.~3, 967--982 (1995).

\bibitem{Manchon}
	D.~Manchon,
	{\sl{A short survey on pre-Lie algebras}}, 
	E.~Schr\"odinger Institut Lectures in Math. Phys., ``Noncommutative Geometry and Physics: 
	Renormalisation, Motives, Index Theory", Eur. Math. Soc, A.~Carey Ed. (2011).
	
\bibitem{MK1}
	H.~Munthe--Kaas,
	{\textsl{Lie--Butcher theory for Runge--Kutta methods}},
	BIT {\bf{35}}, 572--587 (1995).

\bibitem{MK2}
	H.~Munthe--Kaas,
	{\textsl{Runge--Kutta methods on Lie groups}},
	BIT {\bf{38}}, 92--111 (1998).

\bibitem{MielPleb}
    	B.~Mielnik, J.~Pleba\'nski,
    	{\textsl{Combinatorial approach to Baker--Campbell--Hausdorff exponents}}
    	Ann.~Inst.~Henri Poincar\'e A {\bf{XII}}, 215--254 (1970).

\bibitem{Murua}
	A.~Murua, 
	{\textsl{The Hopf algebra of rooted trees, free Lie algebras, and Lie series\/}}, 
	Found.~Comput.~Math.~{\bf{6}}, no 4, 387--426 (2006).

\bibitem{Segal}		
	D.~Segal, 
	{\textsl{Free Left-Symmetric Algebras and an Analogue of the Poincar\'e--Birkhoff--Witt--Theorem}}, 
	J.~Algebra {\bf{164}}, 750--772 (1994).	
	
\bibitem{Spitzer} 
	F.~Spitzer, 
	{\textit{A combinatorial lemma and its application to probability theory}},
        	Trans. Amer. Math. Soc.~{\bf{82}}, 323--339 (1956).

\bibitem{Strichartz}
    	R.~S.~Strichartz,
    	{\textsl{The Campbell--Baker--Hausdorff--Dynkin formula and solutions of differential equations}},
    	J.~Func. Anal.~{\bf{72}}, 320--345 (1987).
	
\bibitem{Sw69} 
	M.~E.~Sweedler,
       	{\textsl{Hopf algebras}},
       	Benjamin, New-York (1969).

\end{thebibliography}
\end{document}